\documentclass{article}

\usepackage{hyperref}
\usepackage{nomencl}
\usepackage{chngcntr}
\usepackage{amsmath}
\usepackage{amsthm}
\usepackage{amscd}
\usepackage{bbm}
\usepackage{enumerate}
\usepackage{amssymb}
\usepackage{palatino}
\usepackage{amscd}
\usepackage{verbatim}
\usepackage{mathrsfs}
\usepackage{esint}
\usepackage{xcolor}
\usepackage[margin=1.8in]{geometry}

\usepackage[T2A,T1]{fontenc}
\usepackage[utf8]{inputenc}

\DeclareSymbolFont{cyrillic}{T2A}{cmr}{m}{n}
\def\makecyrsymbol#1#2{%
  \begingroup\edef\temp{\endgroup
    \noexpand\DeclareMathSymbol{\noexpand#1}
    {\noexpand\mathalpha}{cyrillic}%
    {\expandafter\expandafter\expandafter
     \calccyr\expandafter\meaning\csname T2A\string#2\endcsname\end}}%
  \temp}
\expandafter\def\expandafter\calccyr\string\char#1\end{#1}

\makecyrsymbol\hardsign\cyrhrdsn
\makecyrsymbol\bighardsign\CYRHRDSN
\makecyrsymbol\iy\cyrery
\makecyrsymbol\bigiy\CYRERY
\makecyrsymbol\softsign\cyrsftsn
\makecyrsymbol\bigsoftsign\CYRSFTSN

\newtheoremstyle{dotless}{}{}{\itshape}{}{\bfseries}{}{ }{}

\newtheorem{thm}{Theorem}
\newtheorem{lem}[thm]{Lemma}

\newtheorem{cor}[thm]{Corollary}
\newtheorem{defn}[thm]{Definition}

\newtheorem{conj}[thm]{Conjecture}

\theoremstyle{dotless}

\counterwithin{equation}{section}
\counterwithin{thm}{section}

\newcommand{\End}{\mathop{\text{End}}}
\newcommand{\Hom}{\mathop{\text{Hom}}}

\newcommand{\im}{\mathrm{im}\,}

\newcommand{\Z}{\mathbb{Z}}
\newcommand{\R}{\mathbb{R}}
\newcommand{\Q}{\mathbb{Q}}
\renewcommand{\C}{\mathbb{C}}
\newcommand{\Nm}{\mathrm{Nm}}

\newcommand{\F}{\mathbb{F}}

\newcommand{\id}{\mathop{\mathrm{id}}}

\newcommand{\N}{\mathbb{N}}

\newcommand{\Gal}{\mathop{\mathrm{Gal}}}

\newcommand{\pfrak}{\mathfrak{p}}
\newcommand{\qfrak}{\mathfrak{q}}
\newcommand{\mfrak}{\mathfrak{m}}

\newcommand{\inj}{\hookrightarrow}
\newcommand{\surj}{\twoheadrightarrow}

\newcommand{\Ind}{\mathrm{Ind}}

\newcommand{\triv}{\mathop{\mathrm{triv}}}

\newcommand{\SL}{\mathrm{SL}}
\newcommand{\GL}{\mathrm{GL}}
\newcommand{\PSL}{\mathrm{PSL}}
\newcommand{\actson}{\curvearrowright}

\newcommand{\mat}[4]{\left(\begin{array}{cc} #1 & #2\\ #3 & #4\end{array}\right)}
\newcommand{\twobytwo}[4]{\mat{#1}{#2}{#3}{#4}}

\newcommand{\iso}{\cong}
\newcommand{\eps}{\varepsilon}

\newcommand{\Frob}{\mathrm{Frob}}

\newcommand{\ofrak}{\mathfrak{o}}

\newcommand{\I}{\mathbb{I}}
\newcommand{\invlim}{\varprojlim}

\newcommand{\disjcup}{\bigsqcup}

\newcommand{\rk}{\mathop{\mathrm{rk}}}

\newcommand{\Tor}{\mathrm{Tor}}

\renewcommand{\o}{\ofrak}

\newcommand{\Cl}{\mathrm{Cl}}

\newcommand{\Qbar}{\overline{\Q}}

\renewcommand{\P}{\mathbb{P}}

\renewcommand{\I}{\mathrm{I}}

\newcommand{\tr}{\mathrm{tr}}
\newcommand{\diag}{\mathrm{diag}}

\newcommand{\PGL}{\mathrm{PGL}}
\newcommand{\Jac}{\mathrm{Jac}}

\newcommand{\GSp}{\mathrm{GSp}}
\newcommand{\Lie}{\mathrm{Lie}}
\newcommand{\Fbar}{\overline{\F}}
\newcommand{\rhobar}{\overline{\rho}}
\newcommand{\nfrak}{\mathfrak{n}}
\renewcommand{\End}{\mathrm{End}}

\newcommand{\twoFone}[4]{\, _2F_1\left(\left.\begin{array}{cc} #1 & #2\\ & #3\end{array}\right| #4 \right)}

\newcommand{\const}{\mathrm{const.}}

\let\phi\varphi

\let\emptyset\varnothing

\makeatletter
\let\@@pmod\pmod
\DeclareRobustCommand{\pmod}{\@ifstar\@pmods\@@pmod}
\def\@pmods#1{\mkern4mu({\operator@font mod}\mkern 6mu#1)}
\makeatother

 \DeclareFontFamily{U}{wncy}{}
    \DeclareFontShape{U}{wncy}{m}{n}{<->wncyr10}{}
    \DeclareSymbolFont{mcy}{U}{wncy}{m}{n}
    \DeclareMathSymbol{\Sha}{\mathord}{mcy}{"58}

\title{Modularity and effective Mordell I.}

\author{\Large Levent Alp\"{o}ge}

\date{}

\begin{document}

\maketitle

\renewcommand{\abstractname}{Abstract.}

\begin{abstract}
We give an effective proof of Faltings' theorem for curves mapping to Hilbert modular stacks over odd-degree totally real fields.

We do this by giving an effective proof of the Shafarevich conjecture for abelian varieties of $\GL_2$-type over an odd-degree totally real field.

We deduce for example an effective height bound for $K$-points on the curves $C_a : x^6 + 4y^3 = a^2$ ($a\in K^\times$) when $K$ is odd-degree totally real.

(Over $\Qbar$ all hyperbolic hyperelliptic curves admit an \'{e}tale cover dominating $C_1$.)
\end{abstract}

\section{Introduction.}

Faltings' theorem is one of the classic ineffective results in mathematics. His (first) method of proof is roughly to realize, following Parshin, a given curve inside a moduli space, and then to prove a finiteness result for integral points on that moduli space (compactness provides integrality). He is forced to work with the entire moduli space of principally polarized abelian varieties because of a construction of Kodaira that he invokes.

This forces him to consider quite general Galois representations. We will deal only with Galois representations valued in $\GL_2$ by working only with curves mapping to Hilbert modular varieties\footnote{We will use this terminology for the fine moduli spaces, which are, without level structures imposed, not varieties but stacks.} (or, slightly more generally, curves over which there is a non-isotrivial family of $\GL_2$-type abelian varieties). We will further put ourselves in a situation where it is known that there are motives attached to the conjecturally corresponding automorphic forms by working only over odd-degree totally real fields.

In that situation we will prove that all relevant abelian varieties become modular over a computable\footnote{We remind the reader that to say that a quantity is computable, or, equivalently, effectively computable, is to say that there is a Turing machine that terminates on all inputs and which, on input the relevant parameters --- in this case the base number field and bounds on the dimensions and conductors of the relevant abelian varieties --- outputs said quantity.} finite list of odd-degree totally real extensions, and then deduce a height bound on all such abelian varieties using the usual construction of an abelian variety associated to a parallel weight two Hilbert modular eigencuspform over an odd-degree totally real field.

The technique of course further gives an effective determination of the $S$-integral $K$-points on a Hilbert modular variety (or indeed, interpreted suitably, all Hilbert modular varieties of given dimension together) when $K$ is totally real of odd degree. It also generalizes naturally assuming the usual Galois-to-automorphic-to-motivic conjectures.

We conclude by giving an example application, namely to the curves $x^6 + 4y^3 = a^2$ ($a\in \Qbar^\times$ odd-degree totally real), over each of which we use a particular hypergeometric family.

\subsection{Main theorems.}

We prove the following.

\begin{thm}\label{the odd degree theorem}
There is a finite-time algorithm that, on input $(g,K,S)$, with $g\in \Z^+$, $K/\Q$ totally real of odd degree, and $S$ a finite set of places of $K$, outputs the $g$-dimensional abelian varieties $A/K$ which are of $\GL_2$-type over $K$ and which have good reduction outside $S$.
\end{thm}

Let $\o$ be an order in a totally real field $F/\Q$. Write $H_\o/\Q$ for the\footnote{We will use this sloppy notation (one needs slightly more than just the data of the ring $\o$ to produce a Hilbert modular variety) for convenience.} corresponding $\rk_\Z{\o} = [F:\Q]$-dimensional Hilbert modular variety and $\mathcal{H}_\o$ for its canonical integral model.\footnote{We are free to increase $S$ to avoid subtleties with the latter.} It follows that, at least for $K/\Q$ totally real of odd degree, we may compute $\mathcal{H}_\o(\o_{K,S})$ in finite time.\footnote{We are implicitly invoking the polarization estimates of Gaudron-R\'{e}mond (Theorem $1.1$ of their \cite{gaudron-remond}) and the endomorphism estimates of Masser-W\"{u}stholz (the main theorem of their \cite{masser-wustholz-endomorphism-estimate}, though an explicit such estimate can be extracted from Lemma $9.2$ of Gaudron-R\'{e}mond's \cite{gaudron-remond}) in order to deal with the $\o$-linear polarizations that are hidden by the notation.}

\begin{defn}\ \\\indent
Let $C/\Qbar$ be a smooth projective hyperbolic curve. $C/\Qbar$ is \emph{of type $\I$} if and only if it admits a nonconstant map to a Hilbert modular variety.

Let now $K/\Q$ be a totally real field. Let $C/K$ be a smooth projective hyperbolic curve. $C/K$ is \emph{of type $\dot{\I}$} if and only if it admits a nonconstant map defined over a totally real $L/K$ to a Hilbert modular variety.

Finally let $K/\Q$ be a totally real field of odd degree. Let $C/K$ be a smooth projective hyperbolic curve. $C/K$ is \emph{of type $\bigiy$} if and only if it admits a nonconstant map defined over an odd-degree totally real $L/K$ to a Hilbert modular variety.
\end{defn}

Because each $H_\o$ can be compactified to a projective variety by adding finitely many points, when $\o\neq \Z$ it follows that there are infinitely many complete curves with nonconstant maps to $H_\o$.\footnote{(Add level structure to produce curves rather than just stacky curves from this argument.)}  Indeed, by e.g.\ intersecting with general linear subspaces defined over $\Q$ of codimension $\rk_\Z{\o}-1$, it follows that there are curves $C/\Q$ of type $\bigiy$, so that these are nonempty notions.\footnote{There is moreover enough freedom in choosing $\o$, the implicit level structure, and the linear subspace to guarantee that $C(\Q)\neq \emptyset$. This comment is necessary because e.g.\ (canonical models of) Shimura curves associated to indefinite quaternion algebras over $\Q$ are tautologically of type $\bigiy$, since they map to Hilbert modular surfaces, but of course they lack real points, let alone rational points over odd-degree totally real fields.} We are not sure how to check if e.g.\ a given $C/\Qbar$ is \emph{not} of type $\I$.

The definition allows us to state the following.

\begin{cor}\label{the odd degree corollary}
There is a finite-time algorithm that, on input $(K,C/K)$ with $K/\Q$ totally real of odd degree and $C/K$ of type $\bigiy$, outputs $C(K)$.
\end{cor}

This is completely equivalent to the following perhaps clearer reformulation.

\begin{cor}[Corollary \ref{the odd degree corollary}, height version.]\label{the odd degree corollary, height version}
Let $K/\Q$ be a totally real field of odd degree. Let $C/K$ be a curve of type $\bigiy$. Then: there is an effectively computable $\text{\c{s}}_{C,K}\in \Z^+$ depending only on $K$ and $C/K$ such that all $P\in C(K)$ satisfy: $$h(P)\leq \text{\c{s}}_{C,K}.$$
\end{cor}

In other words, for such $C/K$, $C(L)$ is effectively computable for all odd-degree totally real extensions $L/K$.

We emphasize that this is the first time such an effective result has been proven. Even admitting the Chabauty hypothesis (which presumably fails for sufficiently large such $L/K$), etc., one does not know ahead of time that the upper bound on $\#|C(L)|$ computed by $p$-adic methods and suitable sieving will ever match the lower bound computed by searching for $L$-points of larger and larger height. On the other hand, it is certainly the case that existing methods (employing Chabauty, Chabauty-Kim, etc.) are far more practical than the algorithm implicit in Corollary \ref{the odd degree corollary} when they work (namely quite often).

We conclude with an example. Using Theorem \ref{the odd degree theorem} we prove the following.

\begin{thm}\label{the odd degree example}
Let $K/\Q$ be a totally real field of odd degree. Let $a\in K^\times$. Let $C_a : x^6 + 4y^3 = a^2$. Then: there is a finite time algorithm that, on input an odd-degree totally real extension $L/K$, outputs $C_a(L)$.
\end{thm}

Again, this has the following completely equivalent reformulation.

\begin{thm}[Theorem \ref{the odd degree example}, height version.]\label{the odd degree example, height version}
Let $K/\Q$ be a totally real field of odd degree. Let $a\in K^\times$. Let $C_a : x^6 + 4y^3 = a^2$. Then: there is an effectively computable $\bighardsign_{K,a}\in \Z^+$ depending only on $K$ and $a$ such that all $P\in C_a(K)$ satisfy: $$h(P)\leq \bighardsign_{K,a}.$$
\end{thm}

We again emphasize that this is the first time such an effective form of the Mordell conjecture has been proven (note that, given an odd-degree totally real $a\in \Qbar$, the theorem bounds the heights of points in $C_a(L)$ for infinitely many number fields $L/\Q$). If one could remove the phrase "odd-degree totally real" from Theorem \ref{the odd degree example} (i.e.\ in the case of the single curve $x^6 + 4y^3 = 1$ if one could allow all number fields $L/\Q$), by a theorem of Poonen \cite{poonen} generalizing work of Bogomolov-Tschinkel \cite{bogomolov-tschinkel} one would have an effective form of Faltings' theorem for all solvable covers of $\P^1$ over all number fields.\footnote{\label{etale cover footnote}See e.g.\ Chapter $8$ of the author's thesis \cite{my-thesis}. Over $\C$ the curve $x^6 + 4y^3 = 1$ corresponds to the commutator subgroup (which has index $18$) of the arithmetic triangle group $\Delta(3,6,6)$. Similarly, and again over $\C$, the curve $y^2 = x^6 + 1$ used in the aforementioned theorems of Poonen/Bogomolov-Tschinkel and in Chapter $8$ of the author's thesis \cite{my-thesis} corresponds to the commutator subgroup (which has index $12$) of the arithmetic triangle group $\Delta(2,6,6)$. The point is that $\Delta(3,6,6)$ is an index-$2$ subgroup of $\Delta(2,6,6)$, so that, at least over $\Qbar$, there is an \'{e}tale cover of $y^2 = x^6 + 1$ by $x^6 + 4y^3 = 1$ of degree $3 = \frac{2\cdot 18}{12}$. Explicitly, $x^6 + 4y^3 = 1$ can be rewritten as $$\left(x^3 + \frac{2y^3}{x^3}\right)^2 = 4\cdot \left(\frac{y}{x}\right)^6 + 1$$ (just complete the square).}

\subsection{Remarks.}

First, when $K = \Q$ and $g = 1$, Theorem \ref{the odd degree theorem} was proven in explicit form by Murty-Pasten \cite{murty-pasten}. Independently and simultaneously, von K\"{a}nel \cite{von-kanel} proved the same theorem with only the restriction that $K = \Q$, again in explicit form. Were Serre's conjecture known over all e.g.\ odd-degree totally real fields, von K\"{a}nel's argument would easily prove Theorem \ref{the odd degree theorem} in explicit form. So our key insight in proving Theorem \ref{the odd degree theorem} is really in how to get away with using only potential modularity theorems rather than the modularity theorems von K\"{a}nel uses over $\Q$.

Though von K\"{a}nel never observes that one can bound heights of \emph{rational} points on curves by using his main theorem, his bounds on heights of $\GL_2$-type abelian varieties over $\Q$ of given conductor do indeed produce (explicit!) height bounds on $P\in C(\Q)$ for smooth projective hyperbolic curves $C/\Q$ admitting a nonconstant map defined over $\Q$ to a Hilbert modular variety. This observation, namely that there are complete curves in these moduli spaces, and thus curves whose rational points are (unconditionally!) controlled by explicit automorphic forms, was the origin of this work --- note, though, that the former point is also observed in recent independent work of Baldi-Grossi \cite{baldi-grossi} about obstructions to rational points on these moduli spaces.

We believe our use of hypergeometric families associated to \emph{cocompact} triangle groups $\Delta(e_0,e_1,e_\infty)$ (our example in Section \ref{the example section} amounts to such a use for the arithmetic triangle group $\Delta(3,6,6)$) for these Diophantine questions is new --- see Chapter $10$ (and specifically Section $10.2.1$) of the author's thesis \cite{my-thesis} for a more detailed discussion which includes the nonarithmetic triangle groups and uses the new potential modularity theorems \cite{ten-author-paper} over CM fields.

It should be clear that, assuming suitable modularity \emph{and} existence-of-abelian-variety conjectures for Galois representations valued in $\GSp_{2g}(\Q_\ell)$, one can prove in the same way that there is a finite-time algorithm that determines the $S$-integral $K$-points on the Siegel modular variety, i.e.\ that such conjectures imply an effective form of the Shafarevich conjecture. In a similar vein, the main result of \cite{paper-with-brian}\footnote{See also Chapter $7$ of \cite{my-thesis}, which is based on the same work.} implies that the effective Shafarevich conjecture (and thus an effective form of the Mordell conjecture) also follows from standard motivic conjectures (Fontaine-Mazur, Hodge, and Tate). In fact making this latter work unconditional in some cases was the main motivation for this work.

\section{Acknowledgements.}
This article is based on Chapters $9$ and $11$ of the author's Ph.D.\ thesis at Princeton University. I would like to thank both my advisor Manjul Bhargava and Peter Sarnak for their patience and encouragement. I would especially like to thank Professor Sarnak for a discussion which left me determined to find an unconditional result like Theorem \ref{the odd degree theorem}.

I would also like to thank Gurbir Dhillon, Mladen Dimitrov, Tony Feng, Brian Lawrence, Curtis McMullen, Christopher Skinner, Jacob Tsimerman, Akshay Venkatesh, Raphael von K\"{a}nel, and Shou-Wu Zhang for informative discussions.

I thank the National Science Foundation (via their grant DMS-$2002109$) and Columbia University for their support during the pandemic.

\section{Preliminaries.}

We will use the following mostly standard results in the proofs of the main theorems.

\subsection{Faltings' Lemma.}

The first such result is the usual form of Faltings' Lemma (save perhaps the observation that it also works for the rings $\o/\lambda^N$).
\begin{lem}[Faltings, see Satz $5$ of his \cite{faltings}]\label{faltings' lemma}
Let $d\in \Z^+$. Let $N\in \Z^+$. Let $K/\Q$ be a number field and $S$ a finite set of places of $K$. Let $\o$ be an order in the ring of integers of a number field and $\lambda$ be a prime of $\o$ with $\Nm\,{\lambda}\leq N$. Let $T_{K,S,d,N}$ be a finite set of primes of $K$ that are prime to $\Nm\,{\lambda}$ which is disjoint from $S$ such that, for all Galois extensions $L/K$ that are unramified outside $S$ and the primes dividing $\Nm\,{\lambda}$ and of degree $[L:K]\leq N^{2d^2}$, the map $T_{K,S,d,N}\to \Gal(L/K)/\text{conj.}$ via $\pfrak\mapsto \Frob_\pfrak$ is surjective. Let $R := \o_\lambda$ or $\o/\lambda^m$ for some $m\in \N$. Let $\rho, \rho': \Gal(\Qbar/K)\to \GL_d(R)$ be unramified outside $S$ and the primes dividing $\Nm\,{\lambda}$ and such that $\tr(\rho(\Frob_\pfrak)) = \tr(\rho'(\Frob_\pfrak))$ for all $\pfrak\in T_{K,S,d,N}$.

Then: $\tr\circ \rho = \tr\circ \rho'$ on $\Gal(\Qbar/K)$.
\end{lem}

\begin{proof}
The standard proof extends upon noting that, for $M\subseteq R^{\oplus n}$, $\#|M/\lambda|\leq \#|R/\lambda|^n$, as can be seen\footnote{Brian Lawrence points out a far more elegant proof: take the preimage $\tilde{M}$ of $M$ under the evident surjection $\o_\lambda^{\oplus n}\surj R^{\oplus n}$ and then run the usual argument: $\tilde{M}$ is a free $\o_\lambda$-module of rank $\leq n$, now reduce mod $\lambda$, QED.} by e.g.\ computing the $\Tor_1$.

Namely, it evidently suffices to show that the $R$-span of $\im{(\rho\oplus \rho': \Gal(\Qbar/K)\to \GL_d(R)^{\times 2})}$ inside $M_d(R)^{\times 2}$ is in fact spanned by $\bigcup_{\pfrak\in T_{K,S,d,N}} (\rho\oplus \rho')(\Frob_\pfrak)$, where $\Frob_\pfrak\subseteq \Gal(\Qbar/K)$ is the Frobenius conjugacy class of $\pfrak$. To do this one uses Nakayama to reduce mod $\lambda$, after which it follows from the hypothesis on $T_{K,S,d,N}$.
\end{proof}

Note that the usual explicit form of the Chebotarev density theorem gives an explicit $T_{K,S,d,N}$ satisfying the above --- thus e.g.\ we may take $T_{K,S,d,N}$ satisfying the above and so that all $\pfrak\in T_{K,S,d,N}$ satisfy $\Nm\,{\pfrak}\ll_{K,S,d,N} 1$, where the implied constant is explicit.

\subsection{An explicit bound on the height of an isogeny factor of an abelian variety.}

In this section we will prove the following lemma.

\begin{lem}\label{explicit height bound on isogeny factors}
Let $g\in \Z^+$. Let $H\in \Z^+$. Let $N\in \Z^+$. Let $K/\Q$ be a number field. Then: there is an explicit (thus effectively computable) constant $C_{g,K,N,H}\in \Z^+$, depending explicitly and only on $g$, $K$, $N$, and $H$, such that the following holds:
\begin{itemize}
\item Let $B/K$ be an abelian variety with $\dim{B}\leq g$, $h(B)\leq H$, and with good reduction outside the primes of $K$ dividing $N$. Let $A/K$ be an abelian variety with good reduction outside the primes of $K$ dividing $N$ and such that $A/\Qbar$ is a $\Qbar$-isogeny factor of $B/\Qbar$. Then: $h(A)\leq C_{g,K,N,H}$.
\end{itemize}
\end{lem}

The lemma will follow easily from combining three standard facts. We note that the argument we give already occurs in von K\"{a}nel's \cite{von-kanel}.

The first standard fact is the Masser-W\"{u}stholz isogeny estimate, in the explicit form given in Gaudron-R\'{e}mond's \cite{gaudron-remond} (Raynaud's isogeny estimate \cite{raynaud-hauteurs-et-isogenies} would work just as well).

\begin{thm}[Gaudron-R\'{e}mond, see Theorem $1.4$ of their \cite{gaudron-remond}]\label{masser wustholz bound}
Let $A,A'/K$ be $K$-isogenous abelian varieties over a number field $K$. Write $g := \dim{A}$. Then: there is a $K$-isogeny $\phi: A\to A'$ of degree $$\deg{\phi}\leq \left((14g)^{64g^2}\cdot [K:\Q]\cdot \max(h(A), \log{[K:\Q]}, 1)^2\right)^{2^{10} g^3} =: \kappa(A),$$ where $h(A)$ is the Faltings height of $A$ using Faltings' original normalization.

Consequently $|h(A') - h(A)|\leq \frac{1}{2}\log{\kappa(A)}$.
\end{thm}

The second standard fact is Bost's lower bound for the Faltings height of an abelian variety in terms of its dimension.

\begin{thm}[Bost, Gaudron-R\'{e}mond (see Corollary $8.4$ and then paragraph $2.3$ of their \cite{gaudron-remond-periods})]\label{bost lower bound}
Let $A/\Qbar$ be an abelian variety. Then: $$h(A)\geq -\frac{\log{(2\pi^2)}}{2}\cdot \dim{A}.$$
\end{thm}

The third standard fact is a theorem of Silverberg producing an explicit extension over which all endomorphisms of a given abelian variety (and, consequently, its geometric isogeny decomposition) are defined.

\begin{thm}[Silverberg, Grothendieck]\label{everything happens over an explicit finite extension}
Let $g\in \Z^+$. Let $K/\Q$ be a number field. Let $S$ be a finite set of places of $K$. Then: there is an explicit finite Galois extension $K'/K$ depending only on $g$, $K$, and $S$ such that the following holds:
\begin{itemize}
\item Let $A/K$ be a $g$-dimensional abelian variety over $K$ with good reduction outside $S$. Then: $A/K'$ is split semistable and $\End_{K'}(A) = \End_{\Qbar}(A)$.

Consequently its $K'$-isogeny decomposition $A\sim_{K'} \prod_i B_i^{\times n_i}$ into $K'$-simple pairwise non-$K'$-isogenous $B_i/K'$ is such that all $B_i/\Qbar$ are $\Qbar$-simple and pairwise non-$\Qbar$-isogenous.
\end{itemize}
\end{thm}

(The statement about split semistability, which is the only reason the above carries Grothendieck's name, is not necessary for us here --- we include it only for the reader's convenience.)

\begin{proof}
Let $K''/K$ be the compositum of all extensions $L/K$ of degree $[L:K]\leq 10^{10^{10}\cdot g^{10^{10}}}$ that are unramified outside primes dividing $10^{10}!\cdot \prod_{\pfrak\in S} \Nm\,{\pfrak}$. That $K''/K$ is an explicit finite extension follows from Minkowski's proof of the Hermite-Minkowski theorem.

Now, all $g$-dimensional $A/K$ with good reduction outside $S$ have full $10^{10}!$-torsion defined over $K''$, since $K(A[10^{10}!])\inj K''$. Thus by the usual explicit form of Grothendieck's semistable reduction theorem we see that $A/K''$ is semistable. Since the property is local, by passing to an explicit finite Galois extension $K'/K$ containing $K''$ and depending only on $g$, $K$, and $S$, we ensure that $A/K'$ is split semistable.

Finally it similarly follows from a standard observation of Silverberg (see Theorem $2.4$ of her \cite{silverberg}) that $\End_{K''}(A) = \End_{\Qbar}(A) = \End_{K'}(A)$.
\end{proof}

Now let us prove Lemma \ref{explicit height bound on isogeny factors}.

\begin{proof}[Proof of Lemma \ref{explicit height bound on isogeny factors}.]
Replacing $K$ by the explicit $K'/K$ produced by Theorem \ref{everything happens over an explicit finite extension} applied to $A\times B$, by Theorem \ref{everything happens over an explicit finite extension} it suffices without loss of generality to upper bound $h(A)$ for $A/K$ a $K$-isogeny factor of $B/K$. By Poincar\'{e} complete reducibility, there is an abelian variety $C/K$ such that $B\sim_K A\times C$. By Theorem \ref{masser wustholz bound}, $h(A) + h(C) = h(A\times C)\ll_{h(B), [K:\Q], \dim{B}} 1$ with explicit implied constant. By Theorem \ref{bost lower bound}, $h(C)\gg -\dim{C}\geq -\dim{B}$. Thus $h(A)\ll_{h(B), [K:\Q], \dim{B}} 1$ with explicit implied constant.

\end{proof}

\subsection{Decomposition of $\GL_2$-type abelian varieties over the reals.}

The following lemma arises because the abelian varieties $A/K$ we search for may not necessarily be $K$-simple, and modularity only implies that a $K$-simple such abelian variety is a quotient of the Jacobian of a suitable Shimura curve. The lemma shows that this is not an issue because we produce a $d$-th root as such a quotient for some $d\leq \dim{A}$, so we can (and do) just replace the Jacobian with its $(\dim{A})$-th power in the argument.

\begin{lem}\label{reducibility lemma}
Let $K/\Q$ be a number field with a real place $K\inj \R$. Let $A/K$ be an abelian variety of $\GL_2$-type over $K$. Then: there is a $K$-simple abelian variety $B/K$ of $\GL_2$-type over $K$ such that $A\sim_K B^{\times \frac{\dim{A}}{\dim{B}}}$.
\end{lem}

In fact we will prove the following more general lemma.

\begin{lem}\label{gl2 type abelian varieties are either isotypic or else a product of cm abelian varieties}
Let $K/\Q$ be a number field. Let $A/K$ be an abelian variety which is of $\GL_2$-type over $K$. Then: either there is a $K$-simple abelian variety $B/K$ of $\GL_2$-type over $K$ such that $A\sim_K B^{\times \frac{\dim{A}}{\dim{B}}}$, or else there are two $K$-simple abelian varieties $B_1, B_2/K$ which are non-$K$-isogenous and admit sufficiently many complex multiplications over $K$ such that $A\sim_K B_1^{\times \frac{\dim{A}}{2\dim{B_1}}}\times B_2^{\times \frac{\dim{A}}{2\dim{B_2}}}$.
\end{lem}

\begin{proof}
Let $E/\Q$ be a CM field of degree $g := [E:\Q] = \dim{A}$ admitting an embedding $E\inj \End_K^0(A)$. Write, via Poincar\'{e} complete reducibility, $A\sim_K \prod_{i=1}^m B_i^{\times n_i}$ with $B_i/K$ $K$-simple and pairwise non-$K$-isogenous. Write $D_i := \End_K^0(B_i)$, $F_i/\Q$ for the centre of the division algebra $D_i$, and $d_i$ for the index of $D_i$ over $F_i$. Note that by the Albert classification $F_i/\Q$ is CM, and either $d_i = 1$ (i.e.\ $D_i = F_i$) and $[F_i^+:\Q]\,\big\vert \dim{B_i}$, or else $d_i\geq 2$ and $\frac{d_i^2\cdot [F_i:\Q]}{2}\,\big\vert \dim{B_i}$ (moreover $d_i\leq 2$ if $F_i/\Q$ is totally real). Let $j$ be such that $n_j\dim{B_j}\leq n_i\dim{B_i}$ for all $i$ --- thus $n_j\dim{B_j}\leq \frac{g}{m}$. Note that $E\inj \End_K^0(B_i^{\times n_i}) = M_{n_i}(D_i)$ for all $i$ (since $E$ is a field and the identity is in the image). After tensoring up to $\C$ (i.e.\ considering $E\otimes_\Q \C\iso \C^{\oplus g}\inj M_{n_i}(D_i\otimes_\Q \C)\iso M_{d_i n_i}(\C)^{\oplus [F_i:\Q]}$) we find that $g\leq n_id_i\cdot [F_i:\Q]$ --- this forces $g\leq n_i\dim{B_i}$ if either $F_i/\Q$ is totally real or if $d_i\geq 2$, and $g\leq 2n_i\dim{B_i}$ otherwise. It follows that $m\leq 2$. If $m = 1$ then $g = n_1\dim{B_1}$ so $\dim{B_1}\leq d_1\cdot [F_1:\Q]$, whence $d_1\leq 2$ since $\frac{d_1^2\cdot [F_1:\Q]}{2}\,\big\vert \dim{B_1}$ if $d_1\geq 2$, and we are done because quaternion algebras contain many CM quadratic extensions. Otherwise $m = 2$, whence, for all $i$, $g = 2n_i\dim{B_i}$ and moreover $D_i = F_i$ is imaginary CM of degree $[F_i:\Q] = 2\dim{B_i} = \frac{g}{n_i}$. In other words both $B_1/K$ and $B_2/K$ admit sufficiently many complex multiplications over $K$.
\end{proof}

Thus to prove Lemma \ref{reducibility lemma} we must rule out CM factors --- we will do so via the following.

\begin{lem}\label{cm argument}
Let $K/\Q$ be a number field with a real place $K\inj \R$. Let $L/\Q$ be an imaginary CM field. Let $\Phi\subseteq \Hom_{\Q\text{-alg.}}(L,\C)$ be a CM type of $L$. Let $A/K$ be an abelian variety admitting a map $L\inj \End_K^0(A)$. Write $$\Lie(A/\C)\cong \bigoplus_{\sigma\in \Phi} \sigma^{a_\sigma}\oplus \overline{\sigma}^{b_\sigma}$$ as complex representations of $L$ (the base change of $A/K$ taken along $K\inj \R\inj \C$) --- thus $a_\sigma, b_\sigma\in \N$. Then: $a_\sigma = b_\sigma$ for all $\sigma\in \Phi$.
\end{lem}

\begin{proof}
$L$ preserves $\Lie(A/K)$ by hypothesis, so that the traces of $L\actson \Lie(A/\C)$ all lie in $K$. For each $\sigma\in \Phi$, use Minkowski to produce an $x\in \o_L$ with $\sigma(x)$ large and almost purely imaginary, and $\tau(x)$ tiny for every other $\sigma\neq \tau\in \Phi$. Since $\sum_{\tau\in \Phi} a_\tau\cdot \tau(x) + b_\tau\cdot \overline{\tau(x)}\in K\inj \R$, it follows that $a_\sigma = b_\sigma$.
\end{proof}

Lemma \ref{reducibility lemma} follows immediately.

\begin{proof}[Proof of Lemma \ref{reducibility lemma}.]
By Lemma \ref{gl2 type abelian varieties are either isotypic or else a product of cm abelian varieties} it suffices to show that there are no abelian varieties $B/K$ with sufficiently many complex multiplications over $K$. But this follows from Lemma \ref{cm argument} (for such abelian varieties $a_\sigma + b_\sigma = 1$ for all $\sigma\in \Phi$).
\end{proof}

\subsection{The residual images of relevant compatible families of two-dimensional Galois representations.}

The following is an effectivization and generalization of a theorem of Dimitrov. We note that it is also an improvement on Lemma $7.1.3$ of the ten-author paper \cite{ten-author-paper}, which takes as input a compatible family and concludes the same for a Dirichlet density $1$ subset of the primes. It is also a key input in \cite{my-isogeny-estimate}, which proves a "height-free" isogeny estimate for $\GL_2$-type abelian varieties.

\begin{lem}[Cf.\ Propositions $3.1$ and $3.8$ of Dimitrov's \cite{dimitrov}, and Lemma $7.1.3$ of the ten-author paper \cite{ten-author-paper}.]\label{irreducibility after some point}
Let $d\in \Z^+$. Let $K/\Q$ be a number field. Let $N\in \Z^+$. Let $\ell$ be a prime of $\Z$ which is prime to $N$.

Then: there is an explicit (thus effectively computable) constant $C_{d,K,N,\ell}\in \Z^+$, depending explicitly and only on $d$, $K$, $N$, and $\ell$, such that the following holds. Let $p\geq C_{d,K,N,\ell}$ be a prime of $\Z$. Let $E/\Q$ be a number field of degree $[E:\Q]\leq d$. Let $\lambda, \pfrak\subseteq \o_E$ be primes of $\o_E$ with $\lambda\vert (\ell)$ and $\pfrak\vert (p)$. Let $\rho_\lambda, \rho_\pfrak$ be representations
\begin{align*}
\rho_\lambda: \Gal(\Qbar/K) &\to \GL_2(\o_{E,\lambda})
\\ \rho_\pfrak: \Gal(\Qbar/K) &\to \GL_2(\o_{E,\pfrak})
\end{align*}\noindent
satisfying the following:
\begin{itemize}
\item for all $v\vert \infty$ and primes $\mfrak\subseteq \o_K$ of $\o_K$ with $\mfrak\nmid (N\ell p)$ and such that $\Nm\,{\mfrak}\leq C_{d,K,N,\ell}$, $$\tr(\rho_\lambda(\Frob_\mfrak)) = \tr(\rho_\pfrak(\Frob_\mfrak))\in \o_E$$ and $$|\tr(\rho_\lambda(\Frob_\mfrak))|_v\leq 10^{10}\cdot (\Nm\,{\mfrak})^{10^{10}},$$
\item $\rho_\lambda\otimes_{\o_{E,\lambda}} \Qbar_\ell$ is irreducible, unramified at primes not dividing $N\ell$, and not of the form $\Ind_{\Gal(\Qbar/L)}^{\Gal(\Qbar/K)}{(\chi)}$ with $L/K$ quadratic and $\chi: \Gal(\Qbar/L)\to \Qbar_\ell^\times$ a character,
\item $\rho_\pfrak\otimes_{\o_{E,\pfrak}} E_\pfrak$ has conductor dividing $(N)\cdot (p)^\infty$ and is crystalline with Hodge-Tate weights $\{0,-1\}$ under all embeddings $K\inj \Qbar_p$.
\end{itemize}
Then: writing the mod-$\pfrak$ residual representation as $\overline{\rho}_\pfrak := \rho_\pfrak\otimes_{\o_{E,\pfrak}} \o_{E,\pfrak}/\pfrak$, there is a $g\in \GL_2(\o_E/\pfrak)$ and a subfield $\F_q\subseteq \o_E/\pfrak$ such that: $$g\cdot \SL_2(\F_q)\cdot g^{-1}\subseteq \overline{\rho}_\pfrak(\Gal(\Qbar/K))\subseteq (\o_E/\pfrak)^\times\cdot (g\cdot \GL_2(\F_q)\cdot g^{-1}).$$
\noindent
(Thus in particular the Taylor-Wiles hypothesis at $\pfrak$, namely that $$(\overline{\rho}_\pfrak\otimes_{\o_E/\pfrak} \Fbar_p)\vert_{\Gal(\Qbar/K(\zeta_p))}: \Gal(\Qbar/K(\zeta_p))\to \GL_2(\Fbar_p)$$ is irreducible, holds.)
\end{lem}

The last line follows because $\SL_2(\F_p)$ is perfect, so that the composition $$g\cdot \SL_2(\F_p)\cdot g^{-1}\inj \overline{\rho}_\pfrak(\Gal(\Qbar/K))\surj \overline{\rho}_\pfrak(\Gal(\Qbar/K))/\overline{\rho}_\pfrak(\Gal(\Qbar/K(\zeta_p))$$ is trivial (since the latter is surjected upon by $\Gal(K(\zeta_p)/K)\inj \F_p^\times$), and $\SL_2(\F_p)\inj \GL_2(\Fbar_p)$ is evidently irreducible (consider the torus and then the unipotents).

When $K/\Q$ is totally real our argument largely parallels Dimitrov's except for a trick we introduce, which informally amounts to "moving a congruence in a compatible family until it becomes an equality on a Faltings-Serre set".

\begin{proof}
Because $\SL_2(\F_q)$ is its own commutator subgroup, it suffices to show that there is a subfield $\F_q\subseteq \o_E/\pfrak$ such that the corresponding projective representation $\P(\rhobar_\pfrak): \Gal(\Qbar/K)\to \PGL_2(\o_E/\pfrak)$ has image conjugate to a subgroup of $\PGL_2(\o_E/\pfrak)$ which contains $\PSL_2(\F_q)$ and which is contained in $\PGL_2(\F_q)$. By a classical theorem of Dickson \cite[see pages $285$-$286$ of Ch.\ XII]{dickson}, it then suffices to show that $\rhobar_\pfrak$ is absolutely irreducible, is (absolutely) not induced, and that $\P(\rhobar_\pfrak)$ does not have image of size $\ll 1$.

Because we are going to eventually reduce showing that $\rhobar_\pfrak$ is not induced to showing irreducibility over one of an explicit finite set of quadratic extensions of $K$, for clarity we introduce a second parameter $C_{d,K,N,\ell}'\leq C_{d,K,N,\ell}$ which will be taken to be explicitly sufficiently large.

Let $L/\Q$ be a Galois closure of $K/\Q$, let $\widetilde{E}/\Q$ denote a compositum of $L$ and a Galois closure of $E/\Q$, and let $\widetilde{\pfrak}\subseteq \o_{\widetilde{E}}$ be a prime of $\o_{\widetilde{E}}$ with $\widetilde{\pfrak}\vert \pfrak$. (Thus e.g.\ $[\widetilde{E}:\Q]\leq [E:\Q]!\cdot [K:\Q]!$.)

Now let us note a few consequences of crystallineness of $\rho_\pfrak\otimes_{\o_{E,\pfrak}} E_\pfrak$. First, because $p$ is large, $\rho_\pfrak$ is Fontaine-Laffaille with Hodge-Tate weights $\{0,-1\}$ under all embeddings $K\inj \Qbar_p$. Thus, because this property passes to subquotients, and because the Fontaine-Laffaille-module functor is exact, it follows that $\rhobar_\pfrak$ is Fontaine-Laffaille with the given weights.

Let $\qfrak\subseteq \o_K$ with $\qfrak\vert (p)$ be a prime of $\o_K$ above $p$. Write $I_\qfrak\subseteq \Gal(\Qbar/K)$ for the corresponding inertia subgroup (via a chosen $K_\qfrak\inj \Qbar_p$). If $0\to \alpha\to \rhobar_\pfrak\otimes_{\o_E/\pfrak} \Fbar_p\to \beta\to 0$ with $\alpha, \beta: \Gal(\Qbar/K)\to \Fbar_p^\times$ characters, then, because $\alpha$ and $\beta$ are subquotients of $\rhobar_\pfrak\otimes_{\o_E/\pfrak} \Fbar_p$, they too are Fontaine-Laffaille. It moreover follows that there is a partition $\Hom_{\Q\text{-alg.}}(K_\qfrak, \widetilde{E}_{\widetilde{\pfrak}}) = S_\alpha\disjcup S_\beta$ such that $\alpha\vert_{I_\qfrak}$ has weight $-1$ at an embedding $\sigma: K_\qfrak\inj \widetilde{E}_{\widetilde{\pfrak}}$ if and only if $\sigma\in S_\alpha$.

Because they are Fontaine-Laffaille characters (or for example by twisting by Lubin-Tate characters, or else via Theorem $3.4.3$ of Raynaud's \cite{raynaud-pp}), it follows that $\alpha$ and $\beta$ are products of the fundamental characters corresponding to the respective subsets $S_\alpha$ and $S_\beta$:
\begin{align*}
\alpha\vert_{I_\qfrak}&\equiv \prod_{\sigma\in S_\alpha} \sigma \pmod*{\widetilde{\pfrak}},
\\\beta\vert_{I_\qfrak}&\equiv \prod_{\sigma\in S_\beta} \sigma \pmod*{\widetilde{\pfrak}},
\end{align*}
where we are implicitly identifying the inertia subgroup $I_\qfrak^\text{ab.}$ of $\Gal(\Qbar_p/K_\qfrak)^{\text{ab.}}$ with $\o_{K,\qfrak}^\times$ via local class field theory.

Now let us go through the cases in Dickson's classification one by one.

The third case is easily dealt with: if $\P(\rhobar_\pfrak)$ has image of size $\ll 1$, then, writing $K' := \Qbar^{\ker{\P(\rhobar_\pfrak)}}$ (thus $[K':K]\ll 1$), there is a character $\alpha: \Gal(\Qbar/K')\to (\o_E/\pfrak)^\times$ such that $\rhobar_\pfrak(g)\equiv \alpha(g)\cdot \id\pmod*{\pfrak}$ for all $g\in \Gal(\Qbar/K')\subseteq \Gal(\Qbar/K)$ --- thus in particular $\det{\rhobar_\pfrak(g)}\equiv \alpha(g)^2\pmod*{\pfrak}$. In other words, letting $\qfrak'\subseteq \o_{K'}$ with $\qfrak'\vert (p)$ be a prime above $p$, because $e(\qfrak'/p)\leq [K':K]\ll 1$ (and so still $p\gg e(\qfrak'/p)$) the above reasoning gives that $\alpha\vert_{I_{\qfrak'}}$ is a product of fundamental characters with multiplicities at most $e(\qfrak'/p)\ll 1$, contradicting $\chi_p\vert_{I_{\qfrak'}}\equiv \alpha^2\vert_{I_{\qfrak'}}\pmod*{\pfrak}$. So the third case follows.

Now let us show absolute irreducibility. Suppose for the sake of contradiction that $\rhobar_\pfrak\otimes_{\o_E/\pfrak} \Fbar_p$ is reducible. Then there is a number field $E'/E$ of degree $[E':E]\leq 2$, a prime $\pfrak'\subseteq \o_{E'}$ of $\o_{E'}$ with $\pfrak'\vert \pfrak$, and abelian characters $\psi,\psi': \Gal(\Qbar/K)\to (\o_{E'}/\pfrak')^\times$ such that $$\tr\circ \rho_\pfrak\equiv \psi + \psi'\pmod*{\pfrak'}.$$ Let $E''/\Q$ be a compositum of $E'/\Q$ and $\widetilde{E}/\Q$, and let $\pfrak''\subseteq \o_{E''}$ be a prime of $\o_{E''}$ with $\pfrak''\vert \pfrak'$.

By virtue of being subquotients of $\rhobar_\pfrak$, the characters $\psi$ and $\psi'$ have conductors dividing $(\Delta_N)\cdot (p)^\infty$ with $\Delta_N := N^{(10^{10}\cdot [K:\Q]\cdot [E:\Q])!}$. Therefore by class field theory we may also regard them as characters of the narrow ray class group $\Cl^+_{(\Delta_N)\cdot (p)^\infty}(K) := \invlim_{k\in \Z^+} \Cl^+_{(\Delta_N)\cdot (p)^k}(K)$ associated to the conductor $(\Delta_N)\cdot (p)^\infty$. Again by class field theory we have the exact sequence: $$0\to \left(\prod_{\qfrak\vert (p)} \o_{K,\qfrak}^\times\right)/\overline{\{\eta\in \o_K^\times : \eta\equiv 1\pmod*{\Delta_N}, \eta\text{ tot.\ pos.}\}}\to \Cl^+_{(\Delta_N)\cdot (p)^\infty}(K)\to \Cl^+_{(\Delta_N)}(K)\to 0,$$ where $\overline{\{\eta\in \o_K^\times : \eta\equiv 1\pmod*{\Delta_N}, \eta\text{ tot.\ pos.}\}}$ denotes the closure of $\{\eta\in \o_K^\times : \eta\equiv 1\pmod*{\Delta_N}, \eta\text{ tot.\ pos.}\}\subseteq \prod_{\qfrak\vert (p)} \o_{K,\qfrak}^\times$ under the diagonal embedding and in the product topology.

As we saw in our analysis of the consequences of the Fontaine-Laffaille assumption, for every prime $\qfrak\subseteq \o_K$ of $\o_K$ with $\qfrak\vert (p)$, there is a partition $Z_\qfrak\disjcup Z_\qfrak' = \Hom_{\Q\text{-alg.}}(K_\qfrak, E''_{\pfrak''})$ of $\Hom_{\Q\text{-alg.}}(K_\qfrak, E''_{\pfrak''})\subseteq \Hom_{\Q\text{-alg.}}(K, E'') = \Hom_{\Q\text{-alg.}}(K, L)$ such that $\psi\vert_{I_\qfrak}\equiv \prod_{\sigma\in Z_\qfrak} \sigma\pmod*{\pfrak''}$ and $\psi'\vert_{I_\qfrak}\equiv \prod_{\sigma\in Z_\qfrak'} \sigma\pmod*{\pfrak''}$ (after implicitly composing with the local class field theory isomorphism $I_\qfrak^{\text{ab.}}\simeq \o_{K,\qfrak}^\times$).

Let then $Z := \disjcup_{\qfrak\vert (p)} Z_\qfrak\subseteq \Hom_{\Q\text{-alg.}}(K, L)$ and $Z' := \disjcup_{\qfrak\vert (p)} Z_\qfrak'$.

Regarding $\psi$ as a character of $\Cl^+_{(\Delta_N)\cdot (p)^\infty}(K)$ and evaluating at a totally positive unit $\eta\in \o_K^\times$ with $\eta\equiv 1\pmod*{\Delta_N}$, we find that $$1\equiv \prod_{\sigma\in Z} \sigma(\eta)\pmod*{\pfrak''}.$$

Now, by Minkowski the finite-index subgroup $\{\eta\in \o_K^\times : \eta\equiv 1\pmod*{\Delta_N}, \eta\text{ tot.\ pos.}\}\subseteq \o_K^\times$ is generated by elements of explicitly bounded height (in terms of $K$ and $N$). Thus so long as $C_{d,K,N,\ell}'$ is explicitly sufficiently large in terms of $K$ and $N$, which we may and will assume, at least on the generators the congruence $1\equiv \prod_{\sigma\in Z} \sigma(\eta)\pmod*{\pfrak''}$ implies an equality $1 = \prod_{\sigma\in Z} \sigma(\eta)$, and then this equality extends to the whole of $\{\eta\in \o_K^\times : \eta\equiv 1\pmod*{\Delta_N}, \eta\text{ tot.\ pos.}\}$ by multiplicativity. Therefore the homomorphism $\o_K^\times\to \o_L^\times$ via $\eta\mapsto \prod_{\sigma\in Z} \sigma(\eta)$ has finite image, and thus its image lies in the roots of unity of $L$.

Choosing an embedding $i: L\inj \C$, we see that the homomorphism $\o_K^\times\to \R^+$ via $\eta\mapsto \prod_{\sigma\in Z} |(i\circ \sigma)(\eta)|$ is therefore trivial. So let us redo the analysis that one does when classifying algebraic Hecke characters.

Recall that Dirichlet's unit theorem may be rephrased as the statement that: on any torsion-free finite-index subgroup $G\subseteq \o_K^\times$, the only relations among the multiplicative characters $G\to \R^+$ corresponding to embeddings $i\circ \sigma: K\inj \C$ via $\eta\mapsto |(i\circ \sigma)(\eta)|$ are those generated by the norm relation $\prod_{\sigma: K\inj \C} |(i\circ \sigma)\vert_G| = 1$ and the "trivial" relations $|(i\circ \sigma)\vert_G| = |(i\circ \sigma')\vert_G|$ when $i\circ \sigma, i\circ \sigma': K\inj \C$ are complex-conjugate embeddings.

It follows that: either $Z = \emptyset$, $Z = \Hom_{\Q\text{-alg.}}(K,L)$, or else $Z' = Z^{\text{conj.}}$, where $Z^{\text{conj.}}\subseteq \Hom_{\Q\text{-alg.}}(K,L)$ is the set of $\sigma'\in \Hom_{\Q\text{-alg.}}(K,L)$ for which there is a $\sigma\in Z$ for which $i\circ \sigma$ and $i\circ \sigma'$ are complex-conjugate embeddings $K\inj \C$.

Let us discuss the last of these possibilities, which we will call the "CM case", for a moment. Note that the CM case is independent of the choice of $i$ (by considering $\#|Z|$, for example). Writing $K^{\text{CM}}\subseteq K$ for the maximal CM subfield of $K/\Q$, $L^{\text{CM}}\subseteq L$ for the maximal CM subfield of $L/\Q$, and $H := \Gal(L/L^{\text{CM}})\subseteq \Gal(L/\Q)$ for the normal subgroup (generated by products of pairs of complex conjugations) corresponding to $L^{\text{CM}}\subseteq L$, we see that $H$ fixes $Z$ as a set and that $K^{\text{CM}}/\Q$ and $L^{\text{CM}}/\Q$ are imaginary CM. Further choosing a $\sigma\in Z$, it follows that the set $Z$ is the set of embeddings $K\inj L$ extending an embedding in the uniquely determined\footnote{$\Phi$ is just the set of those embeddings that are restrictions of embeddings in $Z$.} CM type $\Phi\subseteq \Hom_{\Q\text{-alg.}}(K^{\text{CM}}, L)$ of $K^{\text{CM}} = \sigma^{-1}(L^{\text{CM}}\cap \sigma(K))$.

We conclude that in all three cases there are weight $\leq 1$ algebraic Hecke characters $\chi$ and $\chi'$ of $K$ valued in $L$ and thus $E''$ such that the finite-order Galois characters $\psi\cdot \chi^{-1}\pmod*{\pfrak''}$ and $\psi'\cdot \chi'^{-1}\pmod*{\pfrak''}$ are unramified outside primes dividing $N$. Indeed: when $Z = \emptyset$ we take $\chi = \triv$ and $\chi'$ to be the cyclotomic character, when $Z = \Hom_{\Q\text{-alg.}}(K,L)$ we take $\chi$ to be the cyclotomic character and $\chi'$ to be trivial, and in the CM case we take $\chi$, respectively $\chi'$ to be the algebraic Hecke character on $K$ corresponding to the CM type $\Phi$, respectively its complement, on $K^{\text{CM}}$. In other words (absorbing the finite-order characters into the notation) we have concluded that there are weight $\leq 1$ algebraic Hecke characters $\widetilde{\chi}$ and $\widetilde{\chi}'$ of conductor dividing $(\Delta_N)$ and valued in $E''$ such that $$\tr\circ \rhobar_\pfrak\equiv \widetilde{\chi} + \widetilde{\chi}'\pmod*{\pfrak''},$$ where we have abused notation by also using $\widetilde{\chi}$ and $\widetilde{\chi}'$ for the corresponding $\pfrak''$-adic Galois characters.

Now it is time for our trick. It follows that for all primes $\mfrak\subseteq \o_K$ of $\o_K$ such that $\mfrak\nmid (N\ell)$ and such that $\Nm\,{\mfrak}\leq C_{d,K,N,\ell}'$,
\begin{align*}
\tr(\rho_\lambda(\Frob_\mfrak)) &= \tr(\rho_\pfrak(\Frob_\mfrak))
\\&\equiv \widetilde{\chi}(\mfrak) + \widetilde{\chi}'(\mfrak)\pmod*{\pfrak''},
\end{align*}\noindent
so that by comparing sizes of the left- and right-hand sides we see that $\tr(\rho_\lambda(\Frob_\mfrak)) = \widetilde{\chi}(\mfrak) + \widetilde{\chi}'(\mfrak)$ for all such $\mfrak$. By Lemma \ref{faltings' lemma}, upon choosing $C_{d,K,N,\ell}'$ suitably it follows that the representation $\rho_\lambda$ is reducible (into the direct sum of the $\ell$-adic Galois characters corresponding to $\widetilde{\chi}$ and $\widetilde{\chi}'$), a contradiction. We conclude that $\rhobar_\pfrak$ is indeed absolutely irreducible.

It remains only to show that $\rhobar_\pfrak$ is not absolutely induced. We will do this by reducing to the same argument over an explicit finite extension $L/K$ depending only on $K$ and $N$. Suppose for the sake of contradiction that $\rhobar_\pfrak$ is absolutely induced. Write $L/K$ for the corresponding quadratic extension, and $\chi: \Gal(\Qbar/L)\to \Fbar_p^\times$ for the corresponding character (so that $\rhobar_\pfrak\otimes_{\o_E/\pfrak} \Fbar_p\iso \Ind_{\Gal(\Qbar/L)}^{\Gal(\Qbar/K)}(\chi)$).

Note that $L/K$ is automatically unramified outside primes dividing $Np$. We claim that it is in fact unramified at primes above $p$ as well. Indeed, on the one hand we have already seen that, for all primes $\qfrak\subseteq \o_K$ of $\o_K$ with $\qfrak\vert (p)$, $\rhobar_\pfrak\vert_{I_\qfrak}$ is reducible, i.e.\ it admits a stable line, or equivalently $\P(\rhobar_\pfrak(I_\qfrak))\subseteq \PGL_2(\o_E/\pfrak)$ has a fixed point in $\P^1(\o_E/\pfrak)$. On the other hand, were $L/K$ ramified at $\qfrak$ then, writing $\qfrak'\vert \qfrak$ for the prime of $\o_L$ above $\qfrak$, the injection $I_\qfrak/I_{\qfrak'}\inj \Gal(L/K)$ would be an isomorphism. Thus $\Ind_{\Gal(\Qbar/L)}^{\Gal(\Qbar/K)}(\chi)\big\vert_{I_\qfrak}\simeq \Ind_{I_{\qfrak'}}^{I_\qfrak}(\chi\vert_{I_\qfrak'})$ (for example by Mackey). Thus in the canonical basis there is a lift $\sigma\in I_\qfrak$ of the nontrivial element of $\Gal(L/K)$ which acts by $\twobytwo{0}{\chi(\sigma^2)}{1}{0}$, while a $g\in I_{\qfrak'}$ acts by $\twobytwo{\chi(g)}{0}{0}{\chi(\sigma^{-1}\cdot g\cdot \sigma)}$, and for these to have a common fixed point in $\P^1(\Fbar_p)$ we must have that $\chi(g) = \chi(\sigma^{-1}\cdot g\cdot \sigma)$ for all $g\in I_{\qfrak'}$. But then $\P(\rhobar_\pfrak)\vert_{I_\qfrak}$ would be trivial on all of $I_{\qfrak'}$, and thus have image of order at most $2$. Consequently we would be able to repeat the argument (which only used that the image of $\P(\rhobar_\pfrak)\vert_{I_\qfrak}$ had bounded exponent) we used to deal with the third case of the Dickson classification to deduce that $p\ll 1$, a contradiction.
So $L/K$ is indeed unramified at $\qfrak$ for all primes $\qfrak\subseteq \o_K$ of $\o_K$ with $\qfrak\vert (p)$.

So we see that $L/K$ is a quadratic extension which is unramified outside primes dividing $N$. By inspection or by Minkowski's proof of the Hermite-Minkowski theorem it follows that there is an explicit finite set $\mathcal{F}$ of quadratic extensions of $K$ depending only on $K$ and $N$ such that $L\in \mathcal{F}$. We conclude that, were $\rhobar_\pfrak$ induced, then there would be an $L\in \mathcal{F}$ such that $\rhobar_\pfrak\vert_{\Gal(\Qbar/L)}$ would be absolutely reducible. Therefore so long as we choose $C_{d,K,N,\ell} := 1 + \max_{L\in \mathcal{F}} (C_{d,L,N,\ell}')^2$ (the square because the relevant $L/K$ are quadratic), i.e.\ so long as the above absolute irreducibility argument works not only for $K$ but for all $L\in \mathcal{F}$ as well, it follows that $\rhobar_\pfrak$ is not absolutely induced.
\end{proof}

\section{Proof of Theorem \ref{the odd degree theorem}.}

We will now prove Theorem \ref{the odd degree theorem}. Rather than writing out an algorithm in pseudocode and proving its correctness and that it always terminates as in Chapter $9$ of \cite{my-thesis}, we will instead just describe the technique in words in the course of the proof.

For the duration of this section, and for the sake of brevity, we will call an $A/K$ "relevant" if and only if it is a $g$-dimensional abelian variety over $K$ with good reduction outside $S$ and which is of $\GL_2$-type over $K$.

First let us show that it suffices to produce an effectively computable finite set $\mathcal{F}$ of odd-degree totally real extensions such that all relevant $A/K$ are modular over at least one $L\in \mathcal{F}$.

\begin{lem}\label{reduction to potential modularity}
Let $\mathcal{F}$ be a finite set of odd-degree totally real extensions of $K$ such that for all relevant $A/K$ there is an $L\in \mathcal{F}$ such that $A/L$ is modular. Then: there is an explicit (thus effectively computable) constant $C_{g,K,S,\mathcal{F}}\in \Z^+$ depending explicitly and only on $g$, $K$, $S$, and $\mathcal{F}$ such that, for all relevant $A/K$, $h(A)\leq C_{g,K,S,\mathcal{F}}$.
\end{lem}

\begin{proof}
Without loss of generality we may assume $\mathcal{F}$ is a singleton, say $\mathcal{F} = \{L\}$ (just take a maximum). We may further without loss of generality assume $L = K$ since the absolute Faltings height is invariant under base change --- in other words we need only prove such a bound for those relevant $A/K$ which are modular over $K$ itself.

To do this we will show that all relevant $A/K$ which are modular over $K$ are $K$-isogeny factors of the $g$-th power of the Jacobian of an explicit Shimura curve with level structure, though the notation will make the evident proof (modularity produces an "$f$", Hida produces an $A_f$ with matching $L$-function as such a quotient, so since $A$ and $A_f$ have matching degree-two $L$-functions by Faltings there is a nontrivial $K$-map $A_f\to A$ whence $A$ is a $K$-quotient of $B^{\dim{A}}$) quite clumsy.

Let $A/K$ be relevant and modular over $K$. Let $\o$ be an order in a CM field $F/\Q$ such that $[F:\Q] = g$ and $\o\inj \End_K(A)$ (one exists by hypothesis). Let $\ell\in \Z^+$ be an explicitly sufficiently large prime of $\Z$ (so that e.g.\ $\ell$ is prime to the discriminant of $\o$ and all the norms of primes in $S$).

Let, via Lemma \ref{reducibility lemma}, $\widetilde{A}/K$ be $K$-simple, of $\GL_2$-type over $K$, and such that $A\sim_K \widetilde{A}^{\times (\dim{A}/\dim{\widetilde{A}})}$.

For each prime $\lambda\subseteq \o$ of $\o$ with $\lambda\vert (\ell)$, write $\rho_{A,\lambda}: \Gal(\Qbar/K)\to \GL_2(\o_\lambda)$ for the representation corresponding to the $\lambda$-adic Tate module $T_\lambda(A) := \invlim_n A[\lambda^n]$ of $A/K$. Write also $\rho_{A,\ell}: \Gal(\Qbar/K)\to \GL_{2g}(\Z_\ell)$ for the representation corresponding to the $\ell$-adic Tate module $T_\ell(A) := \invlim_n A[\ell^n]$ of $A/K$. Thus e.g.\ $\rho_{A,\ell}\otimes_{\Z_\ell} \Qbar_\ell\simeq \bigoplus_{\lambda\vert (\ell), \sigma: F_\lambda\inj \Qbar_\ell} \rho_{A,\lambda}\otimes_{\o_\lambda, \sigma} \Qbar_\ell$.

Modularity entails the existence of a parallel weight $2$ Hilbert modular eigencuspform $f$ over $K$ and a $\lambda\subseteq \o$ with $\lambda\vert (\ell)$ such that $L(s,\rho_\lambda) = L(s,f)$. Said another way, modularity entails the existence of a parallel weight $2$ Hilbert modular eigencuspform $f$ such that $L(s,f)$ matches one of the degree-two $L$-functions of $A/K$. Let $\nfrak\subseteq \o_K$ be the conductor of $A/K$ (thus $\Nm\,{\nfrak}\ll_{[K:\Q],S} 1$ with explicit implied constant). By the equality of $\eps$-factors in Section $0.5$ of Carayol's \cite{carayol}, it follows that the level of $f$ is $\nfrak$.

Now by a standard construction of Hida (combine Theorem $4.4$ of Hida's \cite{hida} with Jacquet-Langlands transfer) there is an explicit abelian variety $B/K$, namely the Jacobian of an explicit Shimura curve with level structure depending explicitly on $\nfrak$, such that there is a $K$-isogeny factor $A_f/K$ of $B/K$ which is of $\GL_2$-type over $K$ and such that $L(s,f)$ matches one of the degree-two $L$-functions of $L(s,A_f)$. It follows that $T_\ell(A)\otimes_{\Z_\ell} \Qbar_\ell$ and $T_\ell(A_f)\otimes_{\Z_\ell} \Qbar_\ell$ share a nonzero summand, whence $\Hom_K(A,A_f)\otimes_\Z \Qbar_\ell\simeq \Hom_{\Z_\ell[\Gal(\Qbar/K)]}(T_\ell(A), T_\ell(A_f))\otimes_{\Z_\ell} \Qbar_\ell\neq 0$, where we have used Faltings' proof of the Tate conjecture for homomorphisms of abelian varieties. Thus $A/K$ and $A_f/K$ share a nonzero $K$-isogeny factor (automatically of the form $\widetilde{A}^{\times k}$ for some $k\in \Z^+$), whence $A$ is a $K$-isogeny factor of $A_f^{\times g}$, whence also a $K$-isogeny factor of $B^{\times g}$.

Finally an application of Lemma \ref{explicit height bound on isogeny factors} provides the desired explicit height bound.
\end{proof}

So to prove Theorem \ref{the odd degree theorem} it suffices to prove that there is an effectively computable finite set $\mathcal{F}$ of odd-degree totally real extensions of $K$ such that all relevant $A/K$ are modular over some $L\in \mathcal{F}$. This we do by explaining how to effectively compute the extension produced by an argument of Snowden proving Taylor's potential modularity theorem.

\begin{proof}[Proof of Theorem \ref{the odd degree theorem}.]
Let us now follow Snowden's \cite{snowden}. The first thing to note is that Proposition $5.3.1$ of his \cite{snowden} can evidently be improved to a statement asserting the existence of a finite-time algorithm computing, in his notation, such a pair $(F'/F, A/F')$ given an input $(F,K,M,\pfrak,\mathfrak{l},\rho_\pfrak,\rho_{\mathfrak{l}},t_\pfrak,t_{\mathfrak{l}})$ as in the statement of the proposition. Indeed, the moduli space $X$ he produces is a quasiprojective variety (via the usual Satake/Baily-Borel compactification), and he invokes Moret-Bailly's theorem to prove $X$ has a totally real point with certain properties that can easily be checked in finite time. Because $X$ is quasiprojective, say via $X\inj \P^N$, we may simply enumerate the totally real points of $\P^N$ of larger and larger height and degree, and in finite time we will find a suitable point on $X$ by his proof of his Proposition $5.3.1$.\footnote{One should of course be able to be more explicit, but this would require giving an explicit proof of Moret-Bailly's theorem at least in this case.} So Snowden's Proposition $5.3.1$ may be replaced by the stronger statement that there is an effectively computable pair $(F'/F, A/F')$ (corresponding to an effectively computable totally real point on $X$, and effectively computable in terms of the input $(F,K,M,\pfrak,\mathfrak{l},\rho_\pfrak,\rho_{\mathfrak{l}},t_\pfrak,t_{\mathfrak{l}})$) satisfying the conclusions of the proposition.

Now let $N := 10^{10}!\cdot \prod_{\qfrak\in S} (\Nm\,{\qfrak})^{(10^{10}\cdot [K:\Q])!}$. Let $\ell\in \Z^+$ be a prime of $\Z$ with $\ell > N$ and $\ell\ll N$ with explicit implied constant. Let $p\in \Z^+$ be a prime of $\Z$ produced by Lemma \ref{irreducibility after some point} on input $(g,K,N,\ell)$ and such that $p\ll_{g,K,N,\ell} 1$ with explicit implied constant.

Let $q := p^{(10^{10}\cdot g)!}$. Let $\Phi$ be the explicit finite set of (isomorphism classes of) representations $\rhobar: \Gal(\Qbar/K)\to \GL_2(\F_q)$ which are odd, unramified outside primes dividing $Np$, are such that $\det{\rhobar}\cdot \chi_p^{-1}$ is unramified outside primes dividing $N$, and which have image containing a conjugate of $\SL_2(\F_p)\subseteq \GL_2(\F_q)$. (This is an explicit finite set of representations by Minkowski's proof of the Hermite-Minkowski theorem.) Note that all such $\rhobar\in \Phi$ satisfy Snowden's hypothesis "(A1)" (and of course his hypothesis "(A2)" since $p$ is large).

Lemma \ref{irreducibility after some point} amounts to the statement that, for all relevant $A/K$, the mod-$p$ two-dimensional residual representation $\rhobar_{A,\Fbar_p}: \Gal(\Qbar/K)\to \GL_2(\Fbar_p)$ corresponding to $A/K$ is an element of $\Phi$, in the sense that there is a $\rhobar\in \Phi$ with $\rhobar_{A,\Fbar_p}\iso \rhobar\otimes_{\F_q} \Fbar_p$.

It is clear from Snowden's proof (see especially his $(5.5)$) of Theorem $5.1.1$ of his \cite{snowden} that, upon replacing his Proposition $5.3.1$ with the strengthening given above, one produces an effectively computable output $(M',F')$ given the input $(\rhobar,\psi,M,t)$, to use the notation of his Theorem $5.1.1$.

Let then $\mathcal{F}'$ be the effectively computable finite set of totally real Galois extensions $F'$ produced by (the algorithm implicit in) said improvement of Theorem $5.1.1$ applied to all tuples $(\rhobar,\Qbar^{\ker{\rhobar}},t)$ with $\rhobar\in \Phi$ and (using Snowden's terminology) $t$ a definite type function on the primes of $K$ above $p$.

We claim that, for all relevant $A/K$, there is an $L\in \mathcal{F}'$ such that $A/L$ is modular over $L$, and indeed this follows immediately from $(5.6)$ of Snowden's \cite{snowden}. Indeed, by construction there is a $\rhobar\in \Phi$ such that $\rhobar_{A,\Fbar_p}\iso \rhobar\otimes_{\F_q} \Fbar_p$, and so by construction of $\mathcal{F}'$ there is an $L\in \mathcal{F}'$ such that there is a Hilbert modular eigencuspform $\widetilde{f}$ over $L$ with mod-$p$ residual representation $\rhobar$ and type function matching that of $A/L$ at primes of $L$ above $p$ (the type function of an abelian variety is definite by purity). By Snowden's Theorem $4.4.2$ we conclude that $A/L$ is modular.

So we have proven that there is a finite set of totally real Galois extensions $\mathcal{F}'$ such that all relevant $A/K$ are modular over some $L\in \mathcal{F}'$. The only thing left to note is that, just as in Snowden's proof of his Proposition $9.4.1$, if an abelian variety $A/K$ is modular over an $L\in \mathcal{F}'$, then, letting $H\subseteq \Gal(L/K)$ be a $2$-Sylow subgroup (thus solvable), by Langlands' solvable descent for $\GL_2$ \cite{langlands} it follows that $A/K$ is modular over $L^H$.

So we let $\mathcal{F}$ be the effectively computable finite set of extensions $L^H$ with $L\in \mathcal{F}'$ and $H\subseteq \Gal(L/K)$ a $2$-Sylow subgroup. We have therefore shown that all relevant $A/K$ are modular over some $L\in \mathcal{F}$, so we are done by Lemma \ref{reduction to potential modularity}.
\end{proof}

\section{Proof of Corollary \ref{the odd degree corollary}.}\label{the odd degree corollary section}

The proof of Corollary \ref{the odd degree corollary} is exactly the same as Parshin's reduction of the Mordell conjecture to the Shafarevich conjecture on abelian varieties via Kodaira's families: compactness provides an a priori estimate on the conductors of abelian varieties occurring in a given family over a proper curve.

\begin{proof}[Proof of Corollary \ref{the odd degree corollary}.]
Let $H_\o$ be a Hilbert modular variety and $f: C\to H_\o$ a map, without loss of generality defined over $K$ itself, witnessing the fact that $C/K$ is of type $\bigiy$ (such an $H_\o$ and such a map $C\to H_\o$ may be found in finite time by brute force search).

Let $\mathscr{C}/\o_K$ be the minimal proper regular model of $C/K$, and let $T$ be an explicit finite set of primes containing $S$ such that $\mathscr{C}/\o_{K,T}$ and $\mathcal{H}_\o/\o_{K,T}$ (the canonical integral model) are smooth and such that the map $C\to H_\o$ extends to $\mathscr{C}\to \mathcal{H}_\o$ over $\o_{K,T}$.

Thus $C(K)\inj \mathscr{C}(\o_{K,T})\to \mathcal{H}_\o(\o_{K,T})$ is finite-to-one, where we have used properness to produce the first map of sets --- thus e.g.\ for all $P\in C(K)$ the abelian variety $A_P/K$ corresponding to the image of $P$ under $C\to H_\o$ has good reduction outside $T$.

The usual explicit estimate for the difference between the Faltings height and the modular height, aka the height pulled back from the Satake/Baily-Borel compactification, implies that there is an explicit very ample divisor (more properly, height data) $D$ of $C$ which is defined over $K$ and such that, for all $P\in C(K)$, $h_D(P)$ is explicitly bounded in terms of $h(A_P)$, where $h_D$ is the height of $P\in C(K)$ with respect to $D$.

Thus by Theorem \ref{the odd degree theorem} we are done.
\end{proof}

\section{Proof of Theorem \ref{the odd degree example}.\label{the example section}}

We have already explained why there are infinitely many curves of type $\bigiy$, but we found it difficult to produce explicit examples of such curves by following that argument (which is explicit, though evidently only in principle) because we were unable to find sufficiently explicit presentations of Hilbert modular varieties (with suitable level structure) as subvarieties of projective space in the literature.

So we will proceed in a different manner. The following family of examples was reverse-engineered from formula $(17)$ of Deines-Fuselier-Long-Swisher-Tu's \cite{deines-fuselier-long-swisher-tu}.

\begin{proof}[Proof of Theorem \ref{the odd degree example}.]

Let, for $a\in \Qbar^\times$ totally real of odd degree, $C_a : x^6 + 4y^3 = a^2$.

Let, for $\lambda\neq 0,1,\infty$, $X_\lambda/\Q(\lambda)$ be the desingularization of the curve $y^6 = x^4 (1 - x)^3 (1 - \lambda\cdot x)$. Note that each fibre of this family $X\to \P^1 - \{0,1,\infty\}$ is a genus $3$ curve. Let $\Jac\,{X_\lambda}/\Q(\lambda)$ be the corresponding family of Jacobians, and let $A_\lambda/\Q(\lambda)$ be the family of Prym varieties corresponding to the map to the desingularization $E_\lambda : y^2 = x^3 + 16\lambda^2$ of the curve $y^3 = x^4 (1-x)^3 (1 - \lambda\cdot x)$. In other words, let $A_\lambda/\Q(\lambda)$ be the connected component of the identity in the kernel of the map on Jacobians induced by the map $X_\lambda\to E_\lambda$ arising from $(x,y)\mapsto (x,y^2)$ taking $y^6 = x^4(1-x)^3(1-\lambda\cdot x)$ to $y^3 = x^4(1-x)^3(1-\lambda\cdot x)$.

Thanks to the automorphism $(x,y)\mapsto (x, \zeta_6\cdot y)$ of $y^6 = x^4 (1-x)^3 (1-\lambda\cdot x)$, it follows that $\Z[\zeta_3]\inj \End_{\Q(\zeta_3, \lambda)}(A_\lambda)$. Thus each abelian surface $A_\lambda/\Q(\lambda)$ is of $\GL_2(\Z[\zeta_3])$-type over $\Q(\zeta_3, \lambda)$.

Because the relevant triangle group, namely $\Delta(3,6,6)$, is arithmetic, we will get even more. Specifically, we will see that the base change $A_\lambda/\overline{\Q(\lambda)}$ admits quaternionic multiplication by an order in the indefinite quaternion algebra over $\Q$ of discriminant $6$. In fact we will further see that the quaternionic multiplication is defined over $\Q\left(\zeta_3, \lambda, (\lambda\cdot \left(\frac{1-\lambda}{4}\right)^2)^{\frac{1}{6}}\right)$, but let us leave this aside for a moment.

Let us first explain how to form a family of abelian surfaces over our $C_a/\Q(a^2)$ using this construction.

Write, for $P =: (x,y)\in C_a$, $f_a(P) := \frac{x^6}{a^2}$. We claim that the pullback family $P\mapsto A_{f_a(P)}$, which is a priori defined over $C_a - f_a^{-1}\left(\{0, 1, \infty\}\}\right)$, extends to a family over all of $C_a$. This is essentially checked in a number of places (see e.g.\ Section $3.1$ of Cohen-Wolfart's \cite{cohen-wolfart}), but we will prove this in Appendix \ref{explicit extension of the hypergeometric family appendix} in order to be thorough and because the relevant period calculations are quite enjoyable. The point is that, because $P\mapsto f_a(P)$, respectively $P\mapsto 1 - f_a(P)$, has a holomorphic sixth root, respectively a holomorphic cube root, on $C_a(\C)$, the particular Schwarz triangle functions that arise in the period lattice of $X_{f_a(P)}$ may be holomorphically continued to their singular points.

So the extended family $A\to C_a$ is a non-isotrivial family of abelian surfaces over $C_a$ which (by considering the family over $C_a - f_a^{-1}(\{0,1,\infty\})$) is defined over $\Q(a^2)$. Moreover we see that $\Z[\zeta_3]\inj \End_{\Q(\zeta_3, f_a(P))}(A_{f_a(P)})$ for all $P\in C_a$.

Just as in the proof of Corollary \ref{the odd degree corollary}, we deduce that there is an explicit finite set $S$ such that, for all $P\in C_a(K)$, the abelian surface $A_{f_a(P)}/K$ has good reduction outside $S$.\footnote{For example we show in Appendix \ref{explicit extension of the hypergeometric family appendix} that there is a map $C_a\to A_3$ such that, writing the induced family of abelian threefolds as $J\to C_a$, we have that $A_{f_a(P)}$ is an isogeny factor of $J_P$. Now one can argue exactly as in the proof of Corollary \ref{the odd degree corollary}, with the moduli space $A_3$ replacing $H_\o$.}

Now let us show that, at least for $P\in C_a(K)$ with $f_a(P)\not\in \{0,1,\infty\}$, in fact $A_{f_a(P)}/K$ is of $\GL_2$-type over $K$. Of course the relevant endomorphism ring is not $\Z[\zeta_3]\inj \End_{K(\zeta_3)}(A_{f_a(P)})$, but thanks to said endomorphism ring for $p$ large we do find that, letting $\pfrak\subseteq \Z[\zeta_3]$ with $\pfrak\vert (p)$ be a prime of $\Z[\zeta_3]$ above $p$ and writing $\rho_{P,\pfrak}: \Gal(\Qbar/K(\zeta_3))\to \GL_2(\Z[\zeta_3]_\pfrak)$ for the representation corresponding to $T_\pfrak(A_{f_a(P)}) := \invlim_n A[\pfrak^n]$, $\rho_{A_{f_a(P)},p}\simeq \rho_{P,\pfrak}\oplus \sigma\circ \rho_{P,\pfrak}$, where $\sigma\in \Gal(\Q(\zeta_3)/\Q)$ is complex conjugation (acting at the level of coefficients) and as usual $\rho_{A_{f_a(P)}, p}$ is the representation corresponding to $T_p(A_{f_a(P)})$. (This decomposition corresponds to and arises from the decomposition $\Z[\zeta_3]\otimes_\Z \Z_p\simeq \bigoplus_{\qfrak\vert (p)} \Z[\zeta_3]_\qfrak$, specifically via the action of the corresponding idempotents.) So far we have simply expressed that $A_{f_a(P)}/K(\zeta_3)$ is of $\GL_2$-type over $K(\zeta_3)$.

Now it is time to see the arithmeticity of $\Delta(3,6,6)$ at the level of Galois representations. We claim that $(\sigma\circ \rho_{P,\pfrak})\otimes_\Z \Q\iso \rho_{P,\pfrak}\otimes_\Z \Q$, or in other words that the Frobenius traces of $\rho_{P,\pfrak}$ at large primes are rational integers.

By counting points on $X_\lambda$ over finite fields in the usual way one sees that the Frobenius traces of $\rho_{P,\pfrak}$ are given by finite field hypergeometric functions --- specifically, by Proposition $13$ of Deines-Fuselier-Long-Swisher-Tu's \cite{deines-fuselier-long-swisher-tu}, for each large prime $\qfrak\subseteq \o_{K(\zeta_3)}$ of $\o_{K(\zeta_3)}$, writing $q := \Nm\,{\qfrak}$ and $\eta: (\o_{K(\zeta_3)}/\qfrak)^\times\surj \mu_6\subseteq \Z[\zeta_3]^\times\subseteq \o_{K(\zeta_3)}^\times$ for the sextic character such that $\eta(\alpha)\equiv \alpha^{\frac{q-1}{6}}\pmod*{\qfrak}$,
\begin{align*}
&\left\{\tr(\rho_{P,\pfrak}(\Frob_\qfrak)), \sigma(\tr(\rho_{P,\pfrak}(\Frob_\qfrak)))\right\}
\\&= \left\{-\eta(-1)\cdot q\cdot \twoFone{\eta}{\eta^2}{\eta^5}{f_a(P)}_q, -\eta^{-1}(-1)\cdot q\cdot \twoFone{\eta^{-1}}{\eta^{-2}}{\eta^{-5}}{f_a(P)}_q\right\}
\end{align*}\noindent
as sets.

By Proposition $9$ of Deines-Fuselier-Long-Swisher-Tu's \cite{deines-fuselier-long-swisher-tu}, we see that
\begin{align*}
\twoFone{\eta}{\eta^2}{\eta^5}{f_a(P)}_q
&= \eta^3(-1)\cdot \eta(-f_a(P))\cdot \eta^2(1-f_a(P))\cdot \frac{J(\eta^2, \eta^3)}{J(\eta, \eta^4)}\cdot \twoFone{\eta^{-1}}{\eta^{-2}}{\eta^{-5}}{f_a(P)}_q
\\&= \eta(f_a(P))\cdot \eta^2(1 - f_a(P))\cdot \frac{J(\eta^2, \eta^3)}{J(\eta, \eta^4)}\cdot \twoFone{\eta^{-1}}{\eta^{-2}}{\eta^{-5}}{f_a(P)}_q,
\end{align*}\noindent
where $J(\cdot, \cdot)$ denotes the usual Jacobi sum.

Now let us evaluate the ratio of Jacobi sums. Of course $\frac{J(\eta^2, \eta^3)}{J(\eta, \eta^4)} = \frac{G(\eta^2)G(\eta^3)}{G(\eta)G(\eta^4)}$ with $G(\cdot)$ the usual Gauss sum (with respect to a chosen additive character $\o_{K(\zeta_3)}/\qfrak\to \mu_q(\C)$). Now we apply the Hasse-Davenport relation with "base" multiplicative character $\eta$ and "shift" multiplicative character $\eta^3$. We find $G(\eta)G(\eta^4) = -\eta^{-2}(2)G(\eta^2)G(\triv)G(\eta^3)$, where as usual $G(\triv) = -1$. So we conclude that 
\begin{align*}
\twoFone{\eta}{\eta^2}{\eta^5}{f_a(P)}_q &= \eta\left(f_a(P)\cdot \left(\frac{1 - f_a(P)}{4}\right)^2\right)\cdot \twoFone{\eta^{-1}}{\eta^{-2}}{\eta^{-5}}{f_a(P)}_q.
\end{align*}

Finally $f_a(P)\cdot \left(\frac{1-f_a(P)}{4}\right)^2 = \frac{X^6}{a^2}\cdot \left(\frac{Y^3}{a^2}\right)^2 = \left(\frac{XY}{a}\right)^6$ is a sixth power in $K$, thus in $K(\zeta_3)$, thus since $\qfrak$ is large $\eta\left(f_a(P)\cdot \left(\frac{1 - f_a(P)}{4}\right)^2\right) = 1$. So we see by Chebotarev and Brauer-Nesbitt (and Faltings' proof of semisimplicity) that $\rho_{P,\pfrak}\otimes_\Z \Q\iso (\sigma\circ \rho_{P,\pfrak})\otimes_\Z \Q$ as claimed.

Consequently the inclusion $\Q(\zeta_3)\otimes_\Q \Q_p\inj \End_{\Q_p[\Gal(\Qbar/K(\zeta_3))]}(\rho_{A_{f_a(P),p}}\otimes_\Z \Q)$ is \emph{not} an isomorphism. Consequently, by Faltings' proof of the Tate conjecture for endomorphisms of abelian varieties, the inclusion $\Q(\zeta_3)\inj \End_{K(\zeta_3)}^0(A_{f_a(P)})$ is not an isomorphism. But because the $\Q$-subalgebra of $\End_{K(\zeta_3)}^0(A_{f_a(P)})$ fixed by complex conjugation must be of dimension at least $\frac{1}{2}\dim_\Q(\End_{K(\zeta_3)}^0(A_{f_a(P)}))$ (since said involution respects products and can be diagonalized), we see that $\dim_\Q(\End_K^0(A_{f_a(P)}))\geq 2$, where we have used that a $K(\zeta_3)$-endomorphism is fixed by complex conjugation if and only if it descends to a $K$-endomorphism (consider the graph).

We conclude that, for $P\in C_a(K)$ with $f_a(P)\not\in \{0,1,\infty\}$, $A_{f_a(P)}/K$ is of $\GL_2$-type over $K$. We have moreover already seen that there is an explicit finite set $S$ such that each $A_{f_a(P)}/K$ has good reduction outside $S$. We are done by Theorem \ref{the odd degree theorem}.
\end{proof}

\appendix
\section{Explicit extension of the hypergeometric family over the singular points.}\label{explicit extension of the hypergeometric family appendix}
In this section we will explain why the family of abelian surfaces over $C_a - f_a^{-1}(\{0,1,\infty\})$ which we explicitly constructed in Section \ref{the example section} extends to a family of abelian surfaces over the whole of $C_a$. Again, this is more or less checked or asserted in a number of references, but we could not find a treatment to our liking in the literature and anyway the calculations are really a treat. Though we treat the particular case of abelian surfaces of interest to us here, it should be clear how to generalize the below calculations to an arbitrary hypergeometric family associated to the triangle group $\Delta(p,q,r)$ with $\frac{1}{p} + \frac{1}{q} + \frac{1}{r} < 1$ and $p,q,r\in \Z^+$.

We will reuse all notation from Section \ref{the example section}. To see the claim it suffices to work over $\C$. Since the family $\lambda\mapsto E_\lambda$ is isotrivial, it further suffices to show that the pullback family $P\mapsto \Jac\,{X_{f_a(P)}}$ extends to all of $C_a$. This is a map from a punctured curve to the moduli space $A_3$. By the Borel extension theorem (here basically Riemann extension) it extends to a map to the Satake/Baily-Borel compactification $C_a\to A_3^{\mathrm{B.B.}}$. Let us now check at the level of periods that the family does not degenerate and indeed extends holomorphically over those $P$ for which $f_a(P)\in \{0, 1, \infty\}$.

As in Remark $12$ of Archinard's \cite{archinard}, we see that the holomorphic differentials on $X_\lambda$ are given by:
\begin{align*}
\omega_1 &:= \frac{dx}{y} = \frac{dx}{\sqrt[6]{x^4(1-x)^3(1-\lambda\cdot x)}},
\\\omega_4 &:= \frac{x^2 (1-x)^2 dx}{y^4} = \frac{dx}{\sqrt[3]{x^2(1-\lambda\cdot x)^2}},
\\\omega_5 &:= \frac{x^3 (1-x)^2 dx}{y^5} = \frac{dx}{\sqrt[6]{x^2(1-x)^3(1-\lambda\cdot x)^5}}.
\end{align*}

Now let us discuss how to compute the periods of $X_\lambda$. Note that the map $X_\lambda\to \P^1$ induced by $(x,y)\mapsto x$ on the singular model of $X_\lambda$ is unramified outside $\{0,1,\frac{1}{\lambda},\infty\}$. We will occasionally regard this map as $X_\lambda\to X_\lambda/\mu_6\iso \P^1$ --- note also that the map $X_\lambda\to E_\lambda$ via $(x,y)\mapsto (x,y^2)$ at the level of singular models may similarly be regarded as $X_\lambda\to X_\lambda/\mu_2\iso E_\lambda$.

In any case, as is classical, $H_1(X_\lambda, \Z)$ is generated by the various lifts of the Pochhammer cycles $\gamma_{0,1}$ and $\gamma_{\frac{1}{\lambda},\infty}$ between $0$ and $1$ and $\frac{1}{\lambda}$ and $\infty$, respectively. Writing $\widetilde{\gamma}_{0,1}$ and $\widetilde{\gamma}_{\frac{1}{\lambda},\infty}$ for the "principal" such lifts (determined by the choice of the principal branch of the sixth root), all others are given by $\zeta_6^k\cdot \widetilde{\gamma}_{0,1}$ and $\zeta_6^{k'}\cdot \widetilde{\gamma}_{\frac{1}{\lambda},\infty}$. Thus we obtain a set of twelve generators of the rank $6$ free $\Z$-module $H_1(X_\lambda,\Z)$.

This makes it straightforward to compute the periods of $X_\lambda$. Write, for $i\in \{1, 4, 5\}$, $\mu_i := \int_{\widetilde{\gamma}_{0,1}} \omega_i$ and $\nu_i := \int_{\widetilde{\gamma}_{\frac{1}{\lambda},\infty}} \omega_i$. Write $\vec{\mu} := \left(\begin{array}{c} \mu_1 \\ \mu_4 \\ \mu_5\end{array}\right)$ and similarly for $\vec{\nu}$. Then the periods of our chosen basis of holomorphic $1$-forms over $\zeta_6^k\cdot \widetilde{\gamma}_{0,1}$ satisfy:
\begin{align*}
\left(\begin{array}{c} \int_{\zeta_6^k\cdot \widetilde{\gamma}_{0,1}} \omega_1 \\ \int_{\zeta_6^k\cdot \widetilde{\gamma}_{0,1}} \omega_4 \\ \int_{\zeta_6^k\cdot \widetilde{\gamma}_{0,1}} \omega_5\end{array}\right) &= \diag(\zeta_6^{-k}, \zeta_6^{-4k}, \zeta_6^{-5k})\cdot \left(\begin{array}{c} \int_{\widetilde{\gamma}_{0,1}} \omega_1 \\ \int_{\widetilde{\gamma}_{0,1}} \omega_4 \\ \int_{\widetilde{\gamma}_{0,1}} \omega_5\end{array}\right),
\\&= \diag(\zeta_6^{-k}, \zeta_6^{-4k}, \zeta_6^{-5k})\cdot \vec{\mu}
\end{align*}\noindent
and similarly for the periods over $\zeta_6^{k'}\cdot \widetilde{\gamma}_{\frac{1}{\lambda},\infty}$.

We next note that the $\Z$-span of the matrices $\diag(\zeta_6^k, \zeta_6^{4k}, \zeta_6^{5k})$ has basis $\id$, $\diag(\zeta_6, \zeta_6^4, \zeta_6^5)$, $\diag(0, 2, 0)$, and $\diag(0, 2\zeta_6, 0)$, the latter two arising from the relations $$\diag(\zeta_6^2, \zeta_6^8, \zeta_6^{10}) = -\id +\, \diag(\zeta_6, \zeta_6^4, \zeta_6^5) - \diag(0, 2\zeta_6, 0)$$ and $$\diag(\zeta_6^3, \zeta_6^{12}, \zeta_6^{15}) = -\id +\, \diag(0,2,0).$$

Consequently the rank $6$ period lattice of $X_\lambda$ with respect to the chosen basis $(\omega_1, \omega_4, \omega_5)$ of holomorphic $1$-forms on $X_\lambda$ and the chosen generators of $H_1(X_\lambda, \Z)$ is the $\Z$-span inside $\C^3$ of $\vec{\mu}$, $\vec{\nu}$, $\diag(\zeta_6, \zeta_6^4, \zeta_6^5)\cdot \vec{\mu}$, $\diag(\zeta_6, \zeta_6^4, \zeta_6^5)\cdot \vec{\nu}$, $\diag(0,2,0)\cdot \vec{\mu}$, $\diag(0,2,0)\cdot \vec{\nu}$, $\diag(0,2\zeta_6,0)\cdot \vec{\mu}$, and $\diag(0,2\zeta_6,0)\cdot \vec{\nu}$. We of course recognize the $\Z$-span of the last four generators as a copy of the period lattice of the CM elliptic curve $E_\lambda/\C$ --- this also provides the required two relations, namely that\footnote{We will be more precise since the calculations are so pleasant.} $\mu_4, \zeta_6\cdot \mu_4\in \Z[\zeta_6]\cdot \nu_4$, among the eight generators of this rank $6$ lattice.

So we finally find the following $\Z$-basis for the period lattice of $X_\lambda$: $\vec{\mu}$, $\vec{\nu}$, $\diag(\zeta_6, \zeta_6^4, \zeta_6^5)\cdot \vec{\mu}$, $\diag(\zeta_6, \zeta_6^4, \zeta_6^5)\cdot \vec{\nu}$, $\diag(0,2,0)\cdot \vec{\nu}$, and $\diag(0,2\zeta_6,0)\cdot \vec{\nu}$.

Now let us actually calculate $\vec{\mu}$ and $\vec{\nu}$ in terms of hypergeometric functions. To do this it suffices to relate the integrals over Pochhammer cycles to usual integrals over a path, and then to invoke Euler's integral representation of Gauss's hypergeometric functions. Specifically, we have the following identities (the $(1 - e(\alpha))(1 - e(\beta))$ factors are calculated by expanding the relevant differentials around the relevant singular points):
\begin{align*}
\mu_1 = \int_{\widetilde{\gamma}_{0,1}} \omega_1 &= \left(1 - e\left(\frac{2}{3}\right)\right)\left(1 - e\left(\frac{1}{2}\right)\right)\cdot \int_0^1 \omega_1,
\\\mu_4 = \int_{\widetilde{\gamma}_{0,1}} \omega_4 &= \left(1 - e\left(\frac{2}{3}\right)\right)\left(1 - e\left(0\right)\right)\cdot \int_0^1 \omega_4,
\\\mu_5 = \int_{\widetilde{\gamma}_{0,1}} \omega_5 &= \left(1 - e\left(\frac{1}{3}\right)\right)\left(1 - e\left(\frac{1}{2}\right)\right)\cdot \int_0^1 \omega_5,
\end{align*}
\begin{align*}
\nu_1 = \int_{\widetilde{\gamma}_{\frac{1}{\lambda},\infty}} \omega_1 &= \left(1 - e\left(\frac{1}{6}\right)\right)\left(1 - e\left(\frac{4}{3}\right)\right)\cdot \int_{\frac{1}{\lambda}}^\infty \omega_1,
\\\nu_4 = \int_{\widetilde{\gamma}_{\frac{1}{\lambda},\infty}} \omega_4 &= \left(1 - e\left(\frac{2}{3}\right)\right)\left(1 - e\left(\frac{4}{3}\right)\right)\cdot \int_{\frac{1}{\lambda}}^\infty \omega_4,
\\\nu_5 = \int_{\widetilde{\gamma}_{\frac{1}{\lambda},\infty}} \omega_5 &= \left(1 - e\left(\frac{5}{6}\right)\right)\left(1 - e\left(\frac{5}{3}\right)\right)\cdot \int_{\frac{1}{\lambda}}^\infty \omega_5,
\end{align*}\noindent
where as usual $e(x) := e^{2\pi i x}$. So we see that e.g.\ $\mu_4 = 0$.\footnote{After all, $\omega_4$ is the pullback of a differential on $E_\lambda\simeq X_\lambda/\mu_2$ and the map $E_\lambda\to \P^1$ via $(x,y)\mapsto x$ at the level of its singular model is unramified over $1$.} Inserting Euler's integral representation $B(b,c-b)\cdot \twoFone{a}{b}{c}{z} = \int_0^1 x^{b-1}(1-x)^{c-a-b-1}(1-z\cdot x)^{-a}\, dx$ we find:
\begin{align*}
\vec{\mu} = \left(\begin{array}{c} 2(1 - \zeta_3^2)\cdot B\left(\frac{1}{3}, \frac{1}{2}\right)\cdot \twoFone{\frac{1}{6}}{\frac{1}{3}}{\frac{5}{6}}{\lambda}
\\ 0 \\ 2(1 - \zeta_3)\cdot B\left(\frac{2}{3},\frac{1}{2}\right)\cdot \twoFone{\frac{5}{6}}{\frac{2}{3}}{\frac{7}{6}}{\lambda} \end{array}\right).
\end{align*}\noindent
Via the change of variables $x\mapsto \frac{1}{\lambda x}$ and then a suitable application of Euler's integral representation (exactly as in $(15)$ of Deines-Fuselier-Long-Swisher-Tu's \cite{deines-fuselier-long-swisher-tu}), we similarly find that:
\begin{align*}
\vec{\nu} = \left(\begin{array}{c}
(1 - \zeta_6)(1 - \zeta_3)\cdot (-1)^{-\frac{2}{3}}\cdot B\left(\frac{1}{3}, \frac{5}{6}\right)\cdot \lambda^{\frac{1}{6}}\cdot \twoFone{\frac{1}{2}}{\frac{1}{3}}{\frac{7}{6}}{\lambda}
\\(1 - \zeta_3^2)(1 - \zeta_3)\cdot (-1)^{-\frac{2}{3}}\cdot B\left(\frac{1}{3},\frac{1}{3}\right)\cdot \lambda^{-\frac{1}{3}}
\\(1 - \zeta_6^5)(1 - \zeta_3^2)\cdot (-1)^{-\frac{4}{3}}\cdot B\left(\frac{2}{3},\frac{1}{6}\right)\cdot\lambda^{-\frac{1}{6}}\cdot\twoFone{\frac{1}{2}}{\frac{2}{3}}{\frac{5}{6}}{\lambda}
\end{array}\right),
\end{align*}\noindent
where we have used that $\twoFone{0}{\frac{1}{3}}{\frac{2}{3}}{\lambda} = 1$ in the calculation of $\nu_4$.

So we have calculated the period lattice of $X_\lambda$ explicitly.

Now it is time to show that the family of Jacobians does not degenerate as $f_a(P)\to 0$, $1$, or $\infty$. Since we are working over $\C$ it suffices to show the claim for $C_1/\C$ only (via the evident isomorphism $C_a\iso C_1$ over $\C$, which also takes $f_a\mapsto f_1$). So, writing $C_1 : X^6 + 4Y^3 = 1$ for clarity, we have $\lambda = X^6$ and $1 - \lambda  = 4Y^3$. Consequently (since we do not want to discuss branches of $z\mapsto \sqrt[6]{z}$) there is an absolute constant $t\in \Z/6$ such that we may write $\vec{\mu}$ and $\vec{\nu}$ at a $\lambda = f_1(P)$ with $P =: (X,Y)$ as:
\begin{align*}
\vec{\mu} &= \left(\begin{array}{c} 2(1 - \zeta_3^2)\cdot B\left(\frac{1}{3}, \frac{1}{2}\right)\cdot \twoFone{\frac{1}{6}}{\frac{1}{3}}{\frac{5}{6}}{X^6}
\\ 0 \\ 2(1 - \zeta_3)\cdot B\left(\frac{2}{3},\frac{1}{2}\right)\cdot \twoFone{\frac{5}{6}}{\frac{2}{3}}{\frac{7}{6}}{X^6} \end{array}\right),
\\\vec{\nu} &= \left(\begin{array}{c}
(1 - \zeta_6)(1 - \zeta_3)\cdot (-1)^{-\frac{2}{3}}\cdot B\left(\frac{1}{3}, \frac{5}{6}\right)\cdot \zeta_6^t\cdot X\cdot \twoFone{\frac{1}{2}}{\frac{1}{3}}{\frac{7}{6}}{X^6}
\\(1 - \zeta_3^2)(1 - \zeta_3)\cdot (-1)^{-\frac{2}{3}}\cdot B\left(\frac{1}{3},\frac{1}{3}\right)\cdot \zeta_6^{-2t}\cdot X^{-2}
\\(1 - \zeta_6^5)(1 - \zeta_3^2)\cdot (-1)^{-\frac{4}{3}}\cdot B\left(\frac{2}{3},\frac{1}{6}\right)\cdot\zeta_6^{-t}\cdot X^{-1}\cdot\twoFone{\frac{1}{2}}{\frac{2}{3}}{\frac{5}{6}}{X^6}
\end{array}\right).
\end{align*}

In a punctured neighbourhood around each of the points in $f_1^{-1}(\{0, 1, \infty\})$ we have that $\mu_1, \mu_5, \nu_4\neq 0$, so we may change our basis of holomorphic $1$-forms on $X_\lambda$ via $\diag(\mu_1^{-1}, \nu_4^{-1}, \mu_5^{-1})\in \GL_3(\C)$. 

This results in the following $\Z$-basis for an equivalent, aka homothetic, period lattice for $X_\lambda$: $(1,0,1)$, $(\zeta_6, 0, \zeta_6^5)$, $(0, 2, 0)$, $(0, 2\zeta_6, 0)$, and the following two vectors:
\begin{align*}
\left(\begin{array}{c} \const\cdot s_{\frac{1}{6}, \frac{1}{3}; \frac{5}{6}}(X^6) \\ 1 \\ \widetilde{\const}\cdot s_{\frac{5}{6}, \frac{2}{3}; \frac{7}{6}}(X^6)\end{array}\right)\text{ and }
\left(\begin{array}{c} \zeta_6\cdot \const\cdot s_{\frac{1}{6}, \frac{1}{3}; \frac{5}{6}}(X^6) \\ \zeta_6^4 \\ \zeta_6^5\cdot \widetilde{\const}\cdot s_{\frac{5}{6}, \frac{2}{3}; \frac{7}{6}}(X^6)\end{array}\right),
\end{align*}\noindent
where $\const, \widetilde{\const}\in \C^\times$ are absolute constants easily calculated from the above, and $$s_{a,b;c}(z) := z^{1-c}\cdot \frac{\twoFone{a-c+1}{b-c+1}{2-c}{z}}{\twoFone{a}{b}{c}{z}}$$ is the classical Schwarz triangle function.

All that is left to note is that the functions $s_{\frac{1}{6}, \frac{1}{3}; \frac{5}{6}}(f_1(P))$ and $s_{\frac{5}{6}, \frac{2}{3}; \frac{7}{6}}(f_1(P))$ extend holomorphically to the whole of $C_1(\C)$ --- indeed this is tautological on $C_1(\C) - f_1^{-1}([1,\infty])$ since we are implicitly using the usual principal branch of the Gauss hypergeometric function in using its definition via Euler's integral over a Pochhammer contour, and then we extend to $f_1^{-1}([1,\infty])$ via the identity
\begin{align*}
\twoFone{a}{b}{c}{z} &= \frac{\Gamma(c)\Gamma(c-a-b)}{\Gamma(c-a)\Gamma(c-b)}\cdot \twoFone{a}{b}{a+b+1-c}{1-z} \\&+ \frac{\Gamma(c)\Gamma(a+b-c)}{\Gamma(a)\Gamma(b)}\cdot (1-z)^{c-a-b}\cdot \twoFone{c-a}{c-b}{1+c-a-b}{1-z}
\end{align*}\noindent
and the fact that $1 - X^6 = 4Y^3$ has a holomorphic cube root on $C_1(\C)$ (note that $c - a - b\in \frac{1}{3}\cdot \Z$ for $(a,b,c)\in \left\{\left(\frac{1}{6}, \frac{1}{3}, \frac{5}{6}\right), \left(\frac{5}{6}, \frac{2}{3}, \frac{7}{6}\right)\right\}$).

So the map $C_a\to A_3^{\mathrm{B.B.}}$ via $P\mapsto \Jac\,{X_{f_a(P)}}$ indeed lands inside $A_3$, whence we have produced a family $J\to C_a$ extending $\lambda\mapsto \Jac\,{X_\lambda}$, and we produce the family $A\to C_a$ extending $\lambda\mapsto A_\lambda$ by taking the fibre above $P\in C_a$ to be the connected component of the identity in the kernel of the map $J_P\to E_P$, where $E\to C_a$ is the evident extension of the isotrivial family $\lambda\mapsto E_\lambda$.

So we are done.

~\\~\\

We will end by making transparent the "extra" endomorphism of $A_{f_1(P)}/\C$ at the level of its period lattice, as in Deines-Fuselier-Long-Swisher-Tu \cite{deines-fuselier-long-swisher-tu} --- after all, $C_1/\C$ corresponds to the commutator subgroup of the triangle group $\Delta(3,6,6)$, which is arithmetic, so $C_1/\C$ must have an interpretation as a Shimura curve with level structure attached to the corresponding quaternion algebra over $\Q$, namely the indefinite quaternion algebra of discriminant $6$. (In those terms our family $P\mapsto A_{f_1(P)}$ is just the corresponding universal family.)

We will project the period lattice to $\C^2$ via $\pi: \C^3\surj \C^2$ via $(x,y,z)\mapsto (x,z)$ since it suffices to work with the period lattice of an isogenous abelian surface. We find the following generators for the relevant period lattice: $\pi(\vec{\mu}), \pi(\vec{\nu}), \diag(\zeta_6, \zeta_6^5)\cdot \pi(\vec{\mu}), \diag(\zeta_6, \zeta_6^5)\cdot \pi(\vec{\nu})$. Via the Euler identity $$\twoFone{a}{b}{c}{z} = (1-z)^{c-a-b}\cdot \twoFone{a-c+1}{b-c+1}{2-c}{z},$$ we see that there is an absolute constant $t'\in \Z/6$ such that:
\begin{align*}
\pi(\vec{\mu}) &= \left(\begin{array}{c} 2(1 - \zeta_3^2)\cdot B\left(\frac{1}{3}, \frac{1}{2}\right)\cdot \twoFone{\frac{1}{6}}{\frac{1}{3}}{\frac{5}{6}}{X^6}
\\ 2(1 - \zeta_3)\cdot B\left(\frac{2}{3},\frac{1}{2}\right)\cdot \twoFone{\frac{5}{6}}{\frac{2}{3}}{\frac{7}{6}}{X^6} \end{array}\right),
\\\pi(\vec{\nu}) &= \left(\begin{array}{c}
(1 - \zeta_6)(1 - \zeta_3)\cdot (-1)^{-\frac{2}{3}}\cdot B\left(\frac{1}{3}, \frac{5}{6}\right)\cdot \zeta_6^{t'}\cdot X\cdot (4^{\frac{1}{3}}\cdot Y)\cdot \twoFone{\frac{5}{6}}{\frac{2}{3}}{\frac{7}{6}}{X^6}
\\(1 - \zeta_6^5)(1 - \zeta_3^2)\cdot (-1)^{-\frac{4}{3}}\cdot B\left(\frac{2}{3},\frac{1}{6}\right)\cdot\zeta_6^{-t'}\cdot X^{-1}\cdot (4^{\frac{1}{3}}\cdot Y)^{-1}\cdot\twoFone{\frac{1}{6}}{\frac{1}{3}}{\frac{5}{6}}{X^6}
\end{array}\right).
\end{align*}

We now claim that the transformation
\begin{align*}
M := \left(\begin{array}{cc}
0
&
\frac{(1 - \zeta_6)(1 - \zeta_3)\cdot (-1)^{-\frac{2}{3}}\cdot B\left(\frac{1}{3}, \frac{5}{6}\right)\cdot \zeta_6^{t'}\cdot X\cdot (4^{\frac{1}{3}}\cdot Y)}{(1 - \zeta_3)\cdot B\left(\frac{2}{3},\frac{1}{2}\right)}
\\
\frac{(1 - \zeta_6^5)(1 - \zeta_3^2)\cdot (-1)^{-\frac{4}{3}}\cdot B\left(\frac{2}{3},\frac{1}{6}\right)\cdot\zeta_6^{-t'}\cdot X^{-1}\cdot (4^{\frac{1}{3}}\cdot Y)^{-1}}{(1 - \zeta_3^2)\cdot B\left(\frac{1}{3}, \frac{1}{2}\right)}
&
0
\end{array}\right)\in \GL_2(\C)
\end{align*}
stabilizes the period lattice --- this is the aforementioned "extra" endomorphism.

Miraculously, $M^2 = 2\cdot \id$ thanks to the equality $$\frac{[(1-\zeta_6)(1-\zeta_6^5)]\cdot [(1-\zeta_3)(1-\zeta_3^2)]\cdot B\left(\frac{2}{3},\frac{1}{6}\right)\cdot B\left(\frac{1}{3},\frac{5}{6}\right)}{[(1-\zeta_3)(1-\zeta_3^2)]\cdot B\left(\frac{1}{3},\frac{1}{2}\right)\cdot B\left(\frac{2}{3},\frac{1}{2}\right)} = 2,$$ the miracle of course being that we have calculated constants sufficiently correctly.

So we may easily check that $M$ stabilizes the period lattice and thus descends to an endomorphism of an abelian variety isogenous to $A_{f_1(P)}/\C$. Indeed by construction $M\cdot \pi(\vec{\mu}) = 2\cdot \pi(\vec{\nu})$, and so $M\cdot \diag(\zeta_6,\zeta_6^5)\cdot \pi(\vec{\mu}) = 2\cdot \diag(\zeta_6^5,\zeta_6)\cdot \pi(\vec{\nu})\in \Z[\diag(\zeta_6, \zeta_6^5)]\cdot \pi(\vec{\nu})$. But then $$M\cdot \pi(\vec{\nu}) = \frac{1}{2}\cdot M^2\cdot \pi(\vec{\mu}) = \pi(\vec{\mu}),$$ and $$M\cdot \diag(\zeta_6,\zeta_6^5)\cdot \pi(\vec{\nu}) = \diag(\zeta_6^5,\zeta_6)\cdot M\cdot \pi(\vec{\nu}) = \diag(\zeta_6^5,\zeta_6)\cdot \pi(\vec{\mu})\in \Z[\diag(\zeta_6,\zeta_6^5)]\cdot \pi(\vec{\mu}).$$

Writing $B_6/\Q$ for the indefinite quaternion algebra over $\Q$ of discriminant $6$, we have therefore produced an explicit inclusion $B_6\inj \End_\C^0(A_{f_1(P)})$ (namely $i\mapsto 1 + 2\zeta_3, j\mapsto M$), concluding this aside and the appendix.

\renewcommand{\refname}{References.}

\bibliographystyle{amsplain}

\bibliography{modularityandmordellnumberone}

\providecommand{\bysame}{\leavevmode\hbox to3em{\hrulefill}\thinspace}
\providecommand{\MR}{\relax\ifhmode\unskip\space\fi MR }
\providecommand{\MRhref}[2]{%
  \href{http://www.ams.org/mathscinet-getitem?mr=#1}{#2}
}
\providecommand{\href}[2]{#2}
\begin{thebibliography}{10}

\bibitem{ten-author-paper}
Patrick~B. Allen, Frank Calegari, Ana Caraiani, Toby Gee, David Helm, Bao V.~Le
  Hung, James Newton, Peter Scholze, Richard Taylor, and Jack~A. Thorne,
  \emph{Potential automorphy over {CM} fields},  (2018), arXiv:1812.09999.

\bibitem{my-isogeny-estimate}
Levent Alp\"{o}ge, \emph{Un peu d'effectivit\'{e} pour les vari\'{e}t\'{e}s
  modulaires de {H}ilbert-{B}lumenthal},  (2021).

\bibitem{paper-with-brian}
Levent Alp\"{o}ge and Brian Lawrence, Forthcoming.

\bibitem{my-thesis}
Levent Hasan~Ali Alp\"{o}ge, \emph{Points on {C}urves}, ProQuest LLC, Ann
  Arbor, MI, 2020, Thesis (Ph.D.)--Princeton University. \MR{4209988}

\bibitem{archinard}
Nat\'{a}lia Archinard, \emph{Hypergeometric abelian varieties}, Canad. J. Math.
  \textbf{55} (2003), no.~5, 897--932. \MR{2005278}

\bibitem{baldi-grossi}
Gregorio Baldi and Giada Grossi, \emph{Finite descent obstruction for {H}ilbert
  modular varieties}, Canad. Math. Bull. \textbf{64} (2021), no.~2, 452--473.
  \MR{4273212}

\bibitem{bogomolov-tschinkel}
Fedor Bogomolov and Yuri Tschinkel, \emph{Unramified correspondences},
  Algebraic number theory and algebraic geometry, Contemp. Math., vol. 300,
  Amer. Math. Soc., Providence, RI, 2002, pp.~17--25. \MR{1936365}

\bibitem{carayol}
Henri Carayol, \emph{Sur les repr\'{e}sentations {$l$}-adiques associ\'{e}es
  aux formes modulaires de {H}ilbert}, Ann. Sci. \'{E}cole Norm. Sup. (4)
  \textbf{19} (1986), no.~3, 409--468. \MR{870690}

\bibitem{cohen-wolfart}
Paula Cohen and J\"{u}rgen Wolfart, \emph{Modular embeddings for some
  nonarithmetic {F}uchsian groups}, Acta Arith. \textbf{56} (1990), no.~2,
  93--110. \MR{1075639}

\bibitem{deines-fuselier-long-swisher-tu}
Alyson Deines, Jenny~G. Fuselier, Ling Long, Holly Swisher, and Fang-Ting Tu,
  \emph{Generalized {L}egendre curves and quaternionic multiplication}, J.
  Number Theory \textbf{161} (2016), 175--203. \MR{3435724}

\bibitem{dickson}
Leonard~Eugene Dickson, \emph{Linear groups, with an exposition of the galois
  field theory}, B.G. Teubner, Leipzig, 1901.

\bibitem{dimitrov}
Mladen Dimitrov, \emph{Galois representations modulo {$p$} and cohomology of
  {H}ilbert modular varieties}, Ann. Sci. \'{E}cole Norm. Sup. (4) \textbf{38}
  (2005), no.~4, 505--551. \MR{2172950}

\bibitem{faltings}
G.~Faltings, \emph{Endlichkeitss\"{a}tze f\"{u}r abelsche {V}ariet\"{a}ten
  \"{u}ber {Z}ahlk\"{o}rpern}, Invent. Math. \textbf{73} (1983), no.~3,
  349--366. \MR{718935}

\bibitem{gaudron-remond}
\'{E}ric Gaudron and Ga\"{e}l R\'{e}mond, \emph{Polarisations et
  isog\'{e}nies}, Duke Math. J. \textbf{163} (2014), no.~11, 2057--2108.
  \MR{3263028}

\bibitem{gaudron-remond-periods}
\bysame, \emph{Th\'{e}or\`eme des p\'{e}riodes et degr\'{e}s minimaux
  d'isog\'{e}nies}, Comment. Math. Helv. \textbf{89} (2014), no.~2, 343--403.
  \MR{3225452}

\bibitem{hida}
Haruzo Hida, \emph{On abelian varieties with complex multiplication as factors
  of the {J}acobians of {S}himura curves}, Amer. J. Math. \textbf{103} (1981),
  no.~4, 727--776. \MR{623136}

\bibitem{langlands}
Robert~P. Langlands, \emph{Base change for {${\rm GL}(2)$}}, Annals of
  Mathematics Studies, No. 96, Princeton University Press, Princeton, N.J.;
  University of Tokyo Press, Tokyo, 1980. \MR{574808}

\bibitem{masser-wustholz-endomorphism-estimate}
D.~W. Masser and G.~W\"{u}stholz, \emph{Endomorphism estimates for abelian
  varieties}, Math. Z. \textbf{215} (1994), no.~4, 641--653. \MR{1269495}

\bibitem{murty-pasten}
M.~Ram Murty and Hector Pasten, \emph{Modular forms and effective {D}iophantine
  approximation}, J. Number Theory \textbf{133} (2013), no.~11, 3739--3754.
  \MR{3084298}

\bibitem{poonen}
Bjorn Poonen, \emph{Unramified covers of {G}alois covers of low genus curves},
  Math. Res. Lett. \textbf{12} (2005), no.~4, 475--481. \MR{2155225}

\bibitem{raynaud-pp}
Michel Raynaud, \emph{Sch\'{e}mas en groupes de type {$(p,\dots, p)$}}, Bull.
  Soc. Math. France \textbf{102} (1974), 241--280. \MR{419467}

\bibitem{raynaud-hauteurs-et-isogenies}
\bysame, \emph{Hauteurs et isog\'{e}nies}, no. 127, 1985, Seminar on arithmetic
  bundles: the Mordell conjecture (Paris, 1983/84), pp.~199--234. \MR{801923}

\bibitem{silverberg}
A.~Silverberg, \emph{Fields of definition for homomorphisms of abelian
  varieties}, J. Pure Appl. Algebra \textbf{77} (1992), no.~3, 253--262.
  \MR{1154704}

\bibitem{snowden}
Andrew Snowden, \emph{On two dimensional weight two odd representations of
  totally real fields},  (2009), arXiv:0905.4266.

\bibitem{von-kanel}
Rafael von K\"{a}nel, \emph{The effective {S}hafarevich conjecture for abelian
  varieties of {$\rm GL_2$}-type}, Forum Math. Sigma \textbf{9} (2021), Paper
  No. e39, 29. \MR{4264210}

\end{thebibliography}

\end{document}